\documentclass[12pt]{article}

\usepackage{multicol} 
\usepackage{times} 
\usepackage{array} 
\usepackage{longtable}

\usepackage{color}
\usepackage{float}
\usepackage[bf,small]{caption2}
\usepackage{comment}
\usepackage{amsmath}
\usepackage{bigints}
\usepackage{rotating, booktabs}
\usepackage{amsopn}
\usepackage{amsfonts}
\usepackage{amsthm}
\usepackage{amssymb}
\usepackage{amsbsy}
\usepackage{gensymb}
\usepackage{multirow}
\usepackage{tabularx}
\usepackage{lscape}
\usepackage{titling}
\usepackage{natbib}
\usepackage{tocbibind}
\usepackage[a4paper,colorlinks,breaklinks,bookmarksopen,bookmarksnumbered]{hyperref}
\usepackage{epsfig,psfrag}
\usepackage{graphicx}
\usepackage{amsmath}
\usepackage{epstopdf}
\usepackage{tabularx}
\usepackage{subfigure}
\usepackage[mathscr]{euscript}
\usepackage[margin=1in]{geometry}
\usepackage{xr-hyper}

\usepackage{amsfonts} 
\usepackage{calc}
\usepackage{titlesec}

\newcommand{\ueq}[1][]{%
  \if\relax\detokenize{#1}\relax
    \sbox0{$\underbrace{=}_{}$}%
    \mathrel{\mathmakebox[\wd0]{=}}
  \else
    \mathrel{\underbrace{=}_{\mathclap{#1}}}
  \fi}

\newcommand {\ctn}{\cite}

\usepackage{url}
\newtheorem{theorem}{Theorem}

\newtheorem{corollary}[theorem]{Corollary}

\newtheorem{remark}[theorem]{Remark}

\begin{document}

\title{\vspace{-0.8in}
\textbf{Bayes Meets Riemann Again: Large Prime Discovery and Re-emergence of the Bone of Contention}}
\author{Durba Bhattacharya$^{\dag}$, Sucharita Roy$^{\dag}$ and Sourabh Bhattacharya$^{\ddag,+}$}
\date{\vspace{-0.5in}}
\maketitle%
\begin{center}
$^{\dag}$ St. Xavier's College\\
$^{\ddag}$ Indian Statistical Institute\\
$+$ Corresponding author:  \href{mailto: bhsourabh@gmail.com}{bhsourabh@gmail.com}
\end{center}


\begin{abstract}
Prime numbers have fascinated mathematicians since antiquity.
Obsession about primes continue even to this day, not only for unravelling
mathematical secrets about their properties, but also for discovering larger and larger prime numbers. The advent of supercomputing facilities have
inspired collaborations among thousands of mathematicians and computer scientists resulting in enormous projects for large prime-hunting. Although
primes form an integral part of encryption/decryption, giant prime numbers are seldom useful in this regard. However, a relatively new field of computer science,
namely, ``Locally Decodable Codes", seems to require applications of very large prime numbers. 


Beneath the practical aspects of prime-hunting, which is largely a brute-force and financially expensive supercomputing exercise, lies the abstract elegance of the 
greatest ever unsolved mathematical conjecture -- the Riemann Hypothesis, intimately tied to the well-established prime number theorem.
The latter establishes prime distribution as essentially random. This perspective
motivates statistical approaches, and the rigour of the Bayesian paradigm, which is also foundationally miles ahead of its frequentiest counterpart, makes it the natural
choice.


We show that the prime number theorem suggests a nonhomogeneous Poisson process for
prime counts, yielding primes as waiting times. We prove that this process is almost surely consistent with the prime
number theorem, asymptotics for large primes, and prime gap results. Building on it, we
develop a fast recursive Bayesian theory for large prime prediction and Riemann Hypothesis
validation, proving it agrees with the traditional but computationally infeasible non-recursive
	Bayesian framework in the limit. Crucially, but not surprisingly (given our past work), our Bayesian approach emphatically falsifies the
Riemann Hypothesis.


We further design a prime-hunting method by simulating from recursive posterior predictive distributions via Transformation-based MCMC. Importantly, 
a simple change-of-variable enables simulation of candidate Mersenne prime exponents. Requiring only modest computing resources, our method has identified 
259 primes exceeding 140 million, of which 184 are strong Mersenne exponent candidates. These correspond to potential Mersenne primes with
42--242 million digits.
\\[2mm]
{\bf Keywords:} 
	{\it Mersenne prime; Nonhomogeneous Poisson process; Prime number theorem; Recursive Bayesian statistics; Riemann Hypothesis; Transformation based 
	Markov Chain Monte Carlo.}
\end{abstract}

\section{Introduction}
\label{sec:intro}

Human curiosity and sense of adventure are the driving forces behind great scientific and mathematical discoveries. It is not necessary that
all such adventures would lead to coveted treasures of great practical values, since the fruits of the adventure of ``thought" need not immediately spill over to the 
practical world. This is often true of pure mathematics, where the practical significance of mathematical discoveries are often not considered important.
Rather, the tree bearing the fruit of discovery, indeed, the method leading to the discovery, is considered to be the holy one, in wistful anticipation of more fruits. 
Agrees G. H. Hardy:{\it ``Pure mathematics is on the whole distinctly more useful than applied. For what is useful above all is technique, and mathematical technique
is taught mainly through pure mathematics."}

Prime numbers, the key ingredients of number theory, have mystified pure mathematicians since time immemorial. A nice account of the obsession for prime numbers
can be found in \ctn{Derby03}, particularly, in connection with Riemann Hypothesis (RH), the greatest ever unsolved conjecture in pure mathematics. 
Somewhat surprisingly, however, although it was known since the ancient ages that the number of prime numbers is infinite, there is an increasing obsession
for discovering larger and larger prime numbers. The efforts for finding larger and larger prime numbers seem to echo a popular Bengali film dialogue: ``{\it janar kono 
sesh naai, janar chesta britha taai}". That is, fruitless is the pursuit of knowledge that is endless!

The craze for large prime numbers, even in the current age, is evident from the announcement of large sum of prize money for discovering large prime numbers
by the Electronic Frontier Foundation (EFF). In $1999$, the discoverer of a prime number with one million digits was awarded 
US\$50,000. In 2008, a ten-million-digit prime did not only win a US\$100,000 prize money, but also a Cooperative Computing Award from the EFF. 
Furthermore, this prime number came to be known as the 29-th top invention of 2008.

Explosion of the computing technology has probably played the major role in contributing to this prime craze. Indeed, both of these large prime numbers have been discovered
by means of tens of thousands of computers and thousands of researchers, in association with a project called the Great Internet Mersenne Prime Search (GIMPS). 
As per the records, 
a further US\$250,000 prize has been announced for the first prime with at least one billion digits, and
a US\$3,000 research discovery award for discovering a new Mersenne prime of less than 100 million digits. Note that a Mersenne prime if of the form $2^n-1$,
for some positive integer $n$. If $n$ is composite, then so is $2^n-1$; hence, equivalently, a Mersenne prime is of the form $2^p-1$, where $p$ is a prime number.

The largest known prime number, as of October $12$, $2024$, is $2^{136,279,841}-1$, which is the $52$-th Mersenne prime number having $41,024,320$ digits.
This was discovered on a network of thousands of dedicated graphics units, supervised by an ex-Nvidia employee, Luke Durant. Interestingly, Durant spent about one year
and 2 million US dollars on this project -- the amount spent far exceeding any prize money announced so far for discovering large prime numbers!

But is the quest for large primes and the offer of large sums of prize money just a nonsensical craze? 
The importance of large prime numbers in secure encryption/decryption is undeniable, after the creation of the RSA (Ron Rivest, 
Adi Shamir and Leonard Adleman) algorithm in $1977$. The key concept driving this algorithm is that, multiplication of two large prime numbers, although
perhaps tedious, produces a large enough number in any reasonable computer, without requiring much computer resources. However, prime factorization of the large
number produced is a highly challenging task, even for the best of computers. The larger the prime numbers, tighter the security. However, the largest prime numbers 
discovered, which are made public, are not used in RSA algorithms as these might possess security threat. But if it becomes possible to generate larger and larger prime numbers
with affordable ease, then of course, ``{\it janar kono sesh naai, janar chesta britha taai}" should perhaps be replaced with ``{\it janar kono sesh naai, tobu janar chesta 
chalao bhai}" (keep on learning brother, even though there is no end to knowledge!)   
Moreover, in the emerging area of computer science ``Locally Decodeable Codes",  new Mersenne primes seem to have immediate applications; see, for example,
\url{http://cgis.cs.umd.edu/~gasarch/TOPICS/pir/threepir.pdf}.

Associated with prime numbers, as already mentioned, is the greatest unsolved mathematical conjecture of all time -- namely, RH.
Although most mathematicians around the word believe it to be true, so far, every attempt to prove the truth of the hypothesis turned out to be erroneous. 
A relatively recent investigation, using characterization of infinite series using Bayesian statistical theory, utterly surprisingly, yielded the negative result, 
that the conjecture is highly likely to be false; see \ctn{SRoy20} and \ctn{Roy23}. Lucid discussion around the result, along with the thoughts of prominent mathematicians
of the past, are provided in \ctn{SB22}. Proposition of a novel and efficient methodology for discovering large primes, that also allows further investigation of 
RH embedded in the same methodology, would be in keeping with G. H. Hardy's philosophy discussed in the beginning -- upholding the methodological 
tree that promises to bear the ripe fruits of success. 

However, we wish to push the philosophy of Hardy a bit further. We do not wish to allow the new method to become trapped within the realms of abstract mathematics, however
elegant. We demand that the methodology discovers new and larger primes extremely fast computationally and without any financial investment at all, unlike the 
case of the latest and largest prime discovered.

With such a goal that attempts to bridge both pure mathematical and practical aspects, we propose a novel Bayesian theory to capture new and large prime numbers
and to address RH, woven into the same theory. Beginning with the basic prime number theorem (PNT), we arrive at a nonhomogeneous Poisson process
for modeling prime counts. The waiting time distribution then describes the distribution of the prime numbers. As we shall point out, the probabilistic concept is inherent
in the basic prime number theory, making statistical approaches most appropriate for prime investigation. 

The philosophy of probabilistic approaches in pure mathematics is likely to infuriate pure mathematicians! However, escape from randomness, even in pure
mathematics, is impossible. Indeed, the number of computable reals is only countable, while there exist uncountably many uncomputable numbers, 
among which, at least the value of the so-called ``halting probability" 
is completely random in the actual probabilistic sense! The number is also {\it maximally unknowable}; see \ctn{SB23} for an account of this discovery by Gregory Chaitin. 
In the words of its discoverer:
``{\it This is a place where God plays dice. I don’t know if any of you remember the dispute many years
ago between Neils Bohr and Albert Einstein about quantum mechanics. Einstein said, `God doesn't
play dice!', and Bohr said, `Well, He does in quantum mechanics!' I think God also plays dice in
pure mathematics}".

Returning to our stochastic approach to prime numbers, it will be seen that the more precise PNT is also associated with RH, which
we incorporate in our nonhomogeneous Poisson process (NHPP) model. We formally establish that our NHPP is almost surely consistent with respect to the PNT; we further 
establish its asymptotic equivalence with large prime numbers and also prove a fundamental asymptotic result related to prime gaps, directly using our NHPP.
We then consider priors on the parameters associated with the basic PNT and RH parts. 
Posterior distributions of these two parameters would then provide reliable validation of the basic PNT and RH, when the conditioned 
number of primes are theoretically taken to infinity. As will be shown, our asymptotic results are considerably robust with respect to prior choice. 

However, given the first $k$ prime numbers, the traditional Bayesian approach would include $2^k$ terms in the posteriors of the parameters and $2^{k+1}$ terms in the 
posterior predictive distribution for capturing new primes.
To illustrate, the posterior predictive distribution 
given the first one million primes, would involve of the order of $2^{\mbox{million}}$ terms, 
which is clearly infeasible to deal with. 
Instead, we create a recursive Bayesian theory where the posterior of the parameters at $k$-the stage is essentially taken as the prior for the same at the $(k+1)$-th
stage. With this recursive prior, the posterior of the parameters for the $(k+1)$-th stage is computed conditional on the $k$-th prime only. This recursive posterior
and the associated recursive posterior predictive distribution consist of only two and four terms, respectively, for any $k$! We further show that these recursive distributions 
asymptotically yield the Same Bayesian inference as the original non-recursive posteriors and posterior predictive distributions.

We show in particular that our Bayesian theory always supports PNT, which is already mathematically established.
But interestingly, and most importantly, RH always turned out to be false with respect to our Bayesian asymptotic approaches. Our results hold for various choices
of the priors, including informative and noninformative, signifying robustness of our approach and results. As we show, our theoretical investigation also allows us to
strengthen the PNT in its general form proven in the past and essentially supports the most precise version of the PNT proven recently.

For the prime-discovery purpose, we 
propose the Transformation based Markov Chain Monte Carlo (TMCMC)  (\ctn{Dutta14}) to generate realizations from the recursive
posterior predictive distribution, suitably approximated and re-parameterized for large primes. In particular, we propose a mixture of additive and multiplicative 
transformations for TMCMC, which is capable of making local moves
as well as large jumps; see \ctn{Dey16} for the details; see also \ctn{Dutta12}. For each realization generated by such TMCMC, we check the primality using an 
efficient prime-checking algorithm 
for Mersenne primes. Any positive result would yield a new and larger prime number. Note that TMCMC convergence diagnostics are unnecessary here, since we are 
not interested in learning the posterior predictive {\it per se}, the closed form of which we are able to derive anyway, but only to detect new and large prime numbers.
Most of the simulations are expected to represent the high-density regions of the recursive posterior predictive distributions, so that the effective sample size is not expected
to be large, for finding new primes. 

Most importantly, a simple change-of-variable of our recursive posterior predictive distribution, facilitates discovering, via TMCMC,
strong candidates for Mersenne prime exponents. In fact, implementation of our theory and methods yielded $184$ potential Mersenne prime exponents, among a pool
of $259$ primes, all greater than $140$ million, which are not yet considered by GIMPS. The number of digits of our potential Mersenne primes 
range from $42,141,405$ to $242,429,718$.

The roadmap of our contributions is given as follows. We begin in Section \ref{sec:motivation} by motivating NHPP for the distribution of prime counts from the PNT.
Formally, we introduce the NHPP and the resultant distribution of the prime numbers in Section \ref{sec:nhpp}. 
In Section \ref{sec:embedding} we formally establish almost sure consistency of our NHPP with respect to some fundamental results on prime numbers. 
In Section \ref{sec:recursive} we introduce our key principles for recursive Bayesian theory for prime prediction and present the details in Section \ref{sec:recursive_details}.
Details of the asymptotic convergence (in probability) of the parameters of our model, with respect to the recursive posterior distributions, with particular focus
on falsification of RH, are provided in Section \ref{sec:con}; in this section we also validate the already established basic PNT.
Emphatic falsification of RH with another version of PNT is detailed in Section \ref{sec:another}.
Using our recursive Bayesian theory, we derive an upper bound for the prime counting function, the details of which are provided in Section \ref{eq:upper_bound}, which is,
in fact, smaller than that given by the traditional PNT. Sharper bounds with our recursive Bayesian theory, depending upon established precise bounds, are discussed
in Section \ref{sec:sharp}. In Section \ref{sec:compare}, we compare relevant bounds approved by our recursive Bayesian theory via an asymptotic model comparison framework
via ratio of recursive posterior predictive densities.
We introduce traditional, non-recursive Bayesian theory for primes in Section \ref{sec:rec_nonrec} and establish its inference-wise asymptotic equivalence with the
recursive Bayesian counterpart. This section also includes establishment of asymptotic equivalence of model selection results based on non-recursive and 
recursive Bayesian theories. In Section \ref{sec:disco_primes} we propose our method of discovering large Mersenne primes using our recursive posterior predictive distributions
and TMCMC, providing details of the experiments and the potential Mersenne primes discovered. Finally, we summarize our contributions and 
make concluding remarks in Section \ref{sec:conclusion}.

\section{The PNT motivation behind NHPP for the prime counts}
\label{sec:motivation}

Let $\varphi(x)$ denote the deterministic prime counting function, that is, the number of primes less than or equal to $x$, 
For more general notation, for any Borel set $A\in\mathfrak B(\mathbb R^+)$, let $\mathcal N(A)$ denote the number of primes in $A$. 
Here $\mathbb R^+=(0,\infty)$ and $\mathfrak B(\mathbb R^+)$ is the
Borel sigma-algebra of $\mathbb R^+$.

Then it is well-known from the prime number theorem (PNT), that as $x\rightarrow\infty$,
\begin{equation*}
	\varphi(x)=\mathcal N((0,x])\sim \frac{x}{\log x},
\end{equation*}
where, for any two functions $f_1(x)$ and $f_2(x)$, $f_1(x)\sim f_2(x)$ stands for $\frac{f_1(x)}{f_2(x)}\rightarrow 1$, as $x\rightarrow\infty$.
It follows that the probability of a prime in $(0,x]$ is
\begin{equation}
	\frac{\varphi(x)}{[x]}=\frac{\mathcal N((0,x])}{[x]}\sim\frac{1}{\log x},
	\label{eq:eq2}
\end{equation}
where $[x]$ denotes the largest integer not exceeding $x$.
Observe that (\ref{eq:eq2}) explicitly shows that probabilistic approaches to dealing with prime numbers, albeit themselves deterministic, are most welcome,
even from the pure mathematical perspective.

Now, for $i=1,\ldots,N$, let
\begin{equation*}
	X_i=\left\{\begin{array}{cc}1 & \mbox{if}~i~\mbox{is prime}\\ 0 & \mbox{otherwise}\end{array}\right.
\end{equation*}
It then follows from (\ref{eq:eq2}) that the mean and variance of $X_i$ are given by
\begin{align}
	&E(X_i)=P(X_i=1)\sim\frac{1}{\log N};\notag\\ 
	&Var(X_i)\sim\frac{1}{\log N}\left(1-\frac{1}{\log N}\right),\notag 
\end{align}
from which it follows that if $X=\sum_{i=1}^NX_i$, then
\begin{align}
	E(X)&\sim\frac{N}{\log N};\notag\\ 
	Var(X)&\sim\frac{N}{\log N}\left(1-\frac{1}{\log N}\right).\notag 
\end{align}

Now let $Y_k$ denote the number of primes between $N_k$ and $N_{k+1}$. Then 
\begin{equation}
	Y_k=\sum_{i=N_k}^{N_{k+1}}X_i\label{eq:eq8}
\end{equation}
and it follows from the above developments that, when $N_k\rightarrow\infty$ as $k\rightarrow\infty$,
\begin{align}
	E(Y_k)&\sim\sum_{i=N_k}^{N_{k+1}}\frac{1}{\log i};\notag\\ 
	Var(Y_k)&\sim\sum_{i=N_k}^{N_{k+1}}\frac{1}{\log i}\left(1-\frac{1}{\log i}\right).\notag 
\end{align}
Now note that
\begin{equation*}
	\left(1-\frac{1}{\log N_{k+1}}\right)\sum_{i=N_k}^{N_{k+1}}\frac{1}{\log i} < \sum_{i=N_k}^{N_{k+1}}\frac{1}{\log i}\left(1-\frac{1}{\log i}\right)
	< \left(1-\frac{1}{\log N_k}\right)\sum_{i=N_{k}}^{N_{k+1}}\frac{1}{\log i}.
\end{equation*}
from which it follows that when $N_k\rightarrow\infty$, as $k\rightarrow\infty$,
\begin{equation}
E(Y_k)\sim Var(Y_k),~\mbox{as}~k\rightarrow\infty.
	\label{eq:po2}
\end{equation}
Also note that, $\sum_{i\in\mathcal S_1}X_i$ and $\sum_{i\in\mathcal S_2}X_i$ are independent whenever the sets $\mathcal S_1$ and $\mathcal S_2$ are disjoint.
This property, in conjunction with (\ref{eq:po2}), strongly suggests a Poisson process model for the prime counts, Since $E(Y_k)$ and $Var(Y_k)$ do not depend only upon
the difference $N_{k+1}-N_k$, the suggested Poisson process is nonhomogeneous. Interestingly and perhaps quite importantly, we establish that the PNT is embedded
in an appropriate nonhomogeneous Poisson process (see Theorems \ref{theorem:embedding1} and \ref{theorem:embedding2}).

%

\subsection{The more precise PNT and RH}
\label{subsec:precise_PNT}
The more precise PNT says that 
\begin{equation}
	\varphi(x)=\int_0^x\frac{1}{\log t}dt+O\left(x^r\log x\right),~\mbox{as}~x\rightarrow\infty,\label{eq:eq11}
\end{equation}
where $r$ is the upper bound of the real parts of the zeros of Riemann's zeta function. It is known that $0\leq r\leq 1$ (see, for example, \ctn{Ingham1932}). 
The integral on the RHS of (\ref{eq:eq11}) is referred to as $Li(x)$ and we denote $\frac{1}{\log x}$ as $li(x)$. 

Importantly, \ctn{Koch1901} proved that if RH is true, then $r=1/2$ in (\ref{eq:eq11}). It follows that if (\ref{eq:eq11}) fails to hold
with $r=1/2$, then RH must be false.

We shall next consider modeling the prime counts as an appropriate Poisson process with intensity measure composed of the $Li$ function and the function
$\sqrt{x}\log x$. Our goal is not only to capture large prime numbers, but also to verify RH in light of the result of \ctn{Koch1901} and to
provide new bounds for the prime counting function, if RH fails to hold with respect to our theory.

\section{An NHPP for prime counts}
\label{sec:nhpp}
Consider an NHPP $\mathcal N(\cdot)$ with intensity measure
\begin{equation}
	\Lambda_{\alpha,\beta}(A)=\alpha\int_A li(u)du+\beta\int_A f(u)du,
	\label{eq:intensity_measure}
\end{equation}
for any set $A\in\mathfrak B\left(\mathbb R^+\right)$, where $\alpha>0$, $\beta>0$. In (\ref{eq:intensity_measure}), in keeping with the error term 
in (\ref{eq:eq11}) and RH, 
we consider the function $f(\cdot)$ to be such that
$$\int_0^x f(t)dt = \sqrt{x}\log x,$$ so that it follows by differentiation that
\begin{equation}
f(x)=\frac{1}{\sqrt{x}}\left(\log\sqrt{x}+1\right).\label{eq:eq12}
\end{equation}

However, we shall set $\mathcal N([0,2))=0$ almost surely, so that the above intensity function will be supported on $t\geq 2$. In other words, we shall
consider an NHPP on $[2,\infty)$. This is because $0$ and $1$ are usually not counted as prime numbers. 
Thus, we re-define the $Li(x)$ function as
\begin{equation}
	Li(x)=\int_2^xli(u)du.
	\label{eq:Li2}
\end{equation}
Note that this definition also removes the singularity at $1$ in the original definition.

The proof of Theorem \ref{theorem:embedding1} shows that $$\mathcal N([2,x])\sim \Lambda_{\alpha,\beta}([2,x]),~\mbox{almost surely, as}~x\rightarrow\infty,$$
showing that $\Lambda_{\alpha,\beta}([2,x])$ is the asymptotic equivalent of the prime counting function outside the null set.
Now, since the difference between the prime counting function and the $Li$ function changes sign infinitely many times 
(\ctn{Little14}, \ctn{Skewes33}), the aforementioned result shows that we must allow $\Lambda_{\alpha,\beta}((0,x])$ to be greater than or less than $Li(x)$ 
infinite number of times, as $x$ is increased. Indeed, as we shall point out subsequently, for the purpose of inference with a novel recursive Bayesian theory,
it would make sense to allow $\alpha$ and $\beta$ to depend upon the stages of the recursion as we proceed sequentially with the ascending order of the prime numbers. 
Thus, for every prime number, these stage-wise parameters are allowed to increase or decrease adaptively to take account of larger or smaller values of the intensity measure
in comparison with the $Li$ function. Even with the traditional, non-recursive Bayesian approach, the posteriors of the parameters induce the same effect
as the number of primes, conditioned upon, increases.

Further, as we shall show, our asymptotic theory does not depend upon the particular form of the function $f(\cdot)$ chosen in the intensity measure, 
but only on the order $O(f(\cdot))$, subject to positivity of the intensity measure. Thus, if the parameters $\alpha$ and $\beta$ converge to $1$ with 
respect to their posterior distributions, in some suitable sense,
then the resultant intensity measure would give the asymptotic upper bound for the intensity measure, which acts as a suitable proxy for the prime counting function.

With (\ref{eq:eq12}), the intensity function for our NHPP is then of the form
\begin{equation}
	\lambda_{\alpha,\beta}(t)=\alpha li(t)+\beta f(t).
	\label{eq:intensity1}
\end{equation}

\subsection{Distribution of the prime numbers as the analogue of the waiting time distribution}
\label{subsec:waiting_time}
Let $Z_1$ denote the waiting time till the first event. Then
\begin{equation}
	P(Z_1>t)=P(\mathcal N([2,t])=0)=\exp\left\{-\Lambda_{\alpha,\beta}([2,t])\right\},\label{eq:w1}
\end{equation}
where
\begin{equation}
	\Lambda_{\alpha,\beta}=\int_0^t\lambda_{\alpha,\beta}(u)du=\alpha\int_2^t\frac{1}{\log u}du+\beta\int_2^t\frac{1}{\sqrt{u}}\left(\log u+1\right)du.\label{eq:w2}
\end{equation}
It follows that the density of $Z_1$ is given by
\begin{align}
	f_{Z_1,\alpha,\beta}(t)&=\exp\left\{-\Lambda_{\alpha,\beta}([2,t])\right\}\Lambda'_{\alpha,\beta}([2,t])\notag\\
	&=\exp\left\{-\Lambda_{\alpha,\beta}([2,t])\right\}\left(\alpha\frac{1}{\log t}+\beta\frac{\log t+1}{\sqrt{t}}\right).\label{eq:z1}
\end{align}
Now let $s\geq 2$, $t>0$ and $\tilde Z_2$ be the inter-arrival time. Then
\begin{align}
	P\left(\tilde Z_2>t|Z_1=s\right)&=P\left(\mathcal N((s,s+t])=0|Z_1=s\right)\notag\\
	&=P\left(\mathcal N((s,s+t]))=0\right)\notag\\
	&=\exp\left\{-\Lambda_{\alpha,\beta}((s,s+t])\right\}.\label{eq:w4}
\end{align}
The density of $\tilde Z_2$ is then given by
\begin{align}
	f_{\tilde Z_2,\alpha,\beta}(t)&=\exp\left\{-\Lambda_{\alpha,\beta}((s,s+t])\right\}\Lambda'_{\alpha,\beta}((s,s+t])\notag\\
	&=\exp\left\{-\Lambda_{\alpha,\beta}((s,s+t])\right\}\left(\alpha\frac{1}{\log (s+t)}+\beta\frac{\log (s+t)+1}{\sqrt{s+t}}\right).\label{eq:z2_tilde}
\end{align}
Now, if $Z_2$ is the waiting time till the second event, given $Z_1=s$, then $Z_2=\tilde Z_2+s$ has the density
\begin{equation}
	f_{Z_2,\alpha,\beta}(t|Z_1=s)=\exp\left\{-\Lambda_{\alpha,\beta}((s,t])\right\}\left(\alpha\frac{1}{\log t}+\beta\frac{\log t+1}{\sqrt{t}}\right);~t>s.\label{eq:z2}
\end{equation}
In general, given $Z_k=t_k$, the density of $Z_{k+1}$ is given by
\begin{equation}
	f_{Z_{k+1},\alpha,\beta}(t|Z_k=t_k)=\exp\left\{-\Lambda_{\alpha,\beta}((t_k,t])\right\}\left(\alpha\frac{1}{\log t}+\beta\frac{\log t+1}{\sqrt{t}}\right);~t>t_k.\label{eq:zk}
\end{equation}

\section{Almost sure consistency of our NHPP with respect to PNT and prime gaps}
\label{sec:embedding}

\begin{theorem}
	\label{theorem:embedding1}
	Our NHPP with $\alpha=1$ and any given $\beta>0$, satisfies
	\begin{equation}
		\mathcal N([2,x])\sim\frac{x}{\log x},~\mbox{almost surely, as}~x\rightarrow\infty.
		\label{eq:embedding1}
	\end{equation}
\end{theorem}
\begin{proof}
	Since $E\left[\mathcal N([2,x])\right] =\Lambda_{\alpha,\beta}([2,x])$, we have, by Markov's inequality, for any $\epsilon>0$,
	\begin{align}
		&P\left(\left|\frac{\mathcal N([2,x])}{\Lambda_{\alpha,\beta}([2,x])}-1\right|>\epsilon\right)<
		\epsilon^{-4}E\left[\frac{\mathcal N([2,x])}{\Lambda_{\alpha,\beta}([2,x])}-1\right]^4\notag\\
		&\qquad =\epsilon^{-4}\frac{\left\{3\Lambda_{\alpha,\beta}([2,x])\right\}^2+\Lambda_{\alpha,\beta}([2,x])}{\left[\Lambda_{\alpha,\beta}([2,x])\right]^4}
		\sim\epsilon^{-4}\frac{3\alpha \left(\log x\right)^2}{x^2},~\mbox{as}~x\rightarrow\infty.
		\label{eq:em1}
	\end{align}
	Now, since
	$$\sum_{n=1}^{\infty}\frac{\left(\log n\right)^2}{n^2}<\infty,$$
	it follows from (\ref{eq:em1}) that
	$$\sum_{n=1}^{\infty}P\left(\left|\frac{\mathcal N([2,n])}{\Lambda_{\alpha,\beta}([2,n])}-1\right|>\epsilon\right)<\infty,$$
	and hence, due to the Borel-Cantelli lemma, it holds that
	$$ \mathcal N([2,x])\sim\Lambda_{\alpha,\beta}([2,x]),~\mbox{almost surely, as}~x\rightarrow\infty.$$
	Since $\Lambda_{\alpha,\beta}([2,x])\sim\frac{\alpha x}{\log x}$, it follows that
	\begin{equation}
		\mathcal N([2,x])\sim\alpha\frac{x}{\log x},~\mbox{almost surely, as}~x\rightarrow\infty.
		\label{eq:em2}
	\end{equation}
	Since (\ref{eq:em2}) holds given any $\alpha>0$, (\ref{eq:embedding1}) holds when $\alpha=1$.
\end{proof}

\begin{corollary}
	\label{cor:cor_pnt0}
	It follows from Theorem \ref{theorem:embedding1} and the PNT that
	our NHPP with $\alpha=1$ and any given $\beta>0$, satisfies
	\begin{equation*}
		\mathcal N([2,x])\sim\varphi(x),~\mbox{almost surely, as}~x\rightarrow\infty.
	\end{equation*}
	\end{corollary}

\begin{remark}
	\label{ewmark:em1}
	We shall show that the posterior distribution of $\alpha$ converges to $1$ in probability, proving that our NHPP is consistent with  the PNT
	in this sense, whenever 
	$$\frac{\hat\beta_x\int_2^xf(u)du}{x/\log x}\rightarrow 0,~\mbox{as}~x\rightarrow\infty,$$
	where $\hat\beta_x$ denotes posterior expectation of $\beta$, to be clarified subsequently.
\end{remark}
	 Theorem \ref{theorem:embedding1} implies that $Z_N\sim N\log N$, almost surely, as $N\rightarrow\infty$, 
	formalized as Theorem \ref{theorem:embedding2}. 

\begin{theorem}
	\label{theorem:embedding2}
	With respect to our NHPP with $\alpha=1$ and any given $\beta>0$, 
	\begin{equation}
		Z_N\sim N\log N,~\mbox{almost surely, as}~N\rightarrow\infty.
		\label{eq:embedding2}
	\end{equation}
\end{theorem}
\begin{proof}
From Theorem \ref{theorem:embedding1} it follows that outside the null set $\mathbb N_1$ (say), 
	\begin{equation}
		\log \left[\mathcal N([2,x])\right]\sim\log x-\log\log x\sim \log x,~\mbox{as}~x\rightarrow\infty.
		\label{eq:pnt1}
	\end{equation}
	Also let $\mathbb N_2$ be the null set outside of which $$Z_N\rightarrow\infty,~\mbox{as}~N\rightarrow\infty.$$
	The following arguments, which essentially constitute the well-known proof in the deterministic prime number scenario, 
	are then valid outside the null set $\mathbb N_1\cup\mathbb N_2$.

		Since $\mathcal N([2,Z_N])=N$, we have 
	\begin{equation}
		\log\left[\mathcal N([2,Z_N])\right]\sim\log N,~\mbox{as}~ N\rightarrow\infty, 
		\label{eq:pnt2}
	\end{equation}
		and applying (\ref{eq:pnt1}) with $x=Z_N$ gives
		\begin{equation}
			\log \left[\mathcal N([2,Z_N])\right]\sim\log Z_N,~\mbox{as}~N\rightarrow\infty.
			\label{eq:pnt3}
		\end{equation}
		Combining (\ref{eq:pnt2}) and (\ref{eq:pnt3}) gives
		\begin{equation}
			\log Z_N\sim\log N,~\mbox{as}~N\rightarrow\infty.
			\label{eq:pnt4}
		\end{equation}
		Again, since $$\frac{Z_N}{\log Z_N}\sim\mathcal N([2,Z_N])=N,$$ as $N\rightarrow\infty$, we have
		$$Z_N\sim N\log(Z_N),$$, which, combining with (\ref{eq:pnt4}) yields
		$$Z_N\sim N\log N,~\mbox{as}~N\rightarrow\infty.$$

		In other words, (\ref{eq:embedding2}) holds.
\end{proof}

\begin{corollary}
	\label{cor:cor_pnt}
	It follows from the PNT that $p_N\sim N\log N$, as $N\rightarrow\infty$, where $p_N$ is the $N$-th prime number. Combining this with Theorem \ref{theorem:embedding2}
	implies that with respect to our NHPP with $\alpha=1$ and any given $\beta>0$, $$Z_N\sim p_N,~\mbox{almost surely, as}~N\rightarrow\infty.$$
\end{corollary}

\begin{remark}
	\label{remark:em3}
	Theorems \ref{theorem:embedding1} and \ref{theorem:embedding2}, which are the stochastic analogues of the PNT, show, in conjunction with 
	Corollary \ref{cor:cor_pnt} that our NHPP is ``consistent" with respect to the PNT in the almost sure sense.
	Again, Remark \ref{ewmark:em1} remains valid in this context.
\end{remark}

\begin{remark}
	\label{remark:em4}
	Theorem \ref{theorem:embedding2} is an encouraging result from our Bayesian perspective, indicating that the Bayesian prediction of $Z_N$, for sufficiently
	large $N$, is expected to be close to the $N$-th prime number $p_N$. More precisely, the posterior predictive distribution of $Z_N$ is expected to include
	$p_N$ in its high-density region. Hence, application of TMCMC, which simulates from high-density regions with high probability, 
	is expected to significantly reduce the effort of discovering large prime numbers.

\end{remark}

\ctn{Hoh1930} proved the important result of number theory that there exits $\theta<1$ such that
\begin{equation}
	\varphi(x+x^\theta)-\varphi(x)\sim\frac{x^\theta}{\log x},~\mbox{as}~x\rightarrow\infty,
	\label{eq:ho1}
\end{equation}
from which it follows that for sufficiently large $N$, the ``prime gaps" $p_{N+1}-p_N$ satisfy
$$p_{N+1}-p_N<p^\theta_N,~\mbox{for sufficiently large}~N.$$

With our NHPP, we establish a stronger result compared to (\ref{eq:ho1}), where $\theta>1/2$. We formalize the result as Theorem \ref{theorem:pgap1}.
\begin{theorem}
	\label{theorem:pgap1}
	Our NHPP with $\alpha=1$ and any given $\beta>0$, satisfies, for $\theta>1/2$,
	\begin{equation}
		\mathcal N([2,x+x^{\theta}])-\mathcal N(]2,x])=\mathcal N((x,x^{\theta}])\sim\frac{x^{\theta}}{\log x},~\mbox{almost surely, as}~x\rightarrow\infty.
		\label{eq:pgap1}
	\end{equation}
\end{theorem}
\begin{proof}
	Note that $E\left[\mathcal N((x,x+x^{\theta}]\right] =\Lambda_{\alpha,\beta}((x,x+x^{\theta})]$. Also, it is easy to verify that for $\theta>1/2$,
	$$\Lambda_{\alpha,\beta}((x,x+x^{\theta}))\sim \frac{x^{\theta}}{\log x},~\mbox{as}~x\rightarrow\infty.$$

	Here we have, by Markov's inequality, for any $\epsilon>0$,
	\begin{align}
		&P\left(\left|\frac{\mathcal N((x,x+x^{\theta}])}{\Lambda_{\alpha,\beta}((x,x+x^{\theta}])}-1\right|>\epsilon\right)<
		\epsilon^{-5}E\left[\frac{\mathcal N((x,x+x^\theta])}{\Lambda_{\alpha,\beta}((x,x+x^\theta])}-1\right]^5\notag\\
		&\qquad =\epsilon^{-5}\frac{\left\{10\Lambda_{\alpha,\beta}((x,x+x^\theta])\right\}^2+\Lambda_{\alpha,\beta}((x,x+x^\theta])}
		{\left[\Lambda_{\alpha,\beta}((x,x+x^\theta)]\right]^5}
		\sim\epsilon^{-4}\frac{10\alpha \left(\log x\right)^3}{x^{3\theta}},~\mbox{as}~x\rightarrow\infty.
		\label{eq:pgap2}
	\end{align}
	Now, since
	$$\sum_{n=1}^{\infty}\frac{\left(\log n\right)^3}{n^{3\theta}}<\infty~\mbox{for}~\theta>1/2,$$
	it follows from (\ref{eq:pgap2}) that
	$$\sum_{n=1}^{\infty}P\left(\left|\frac{\mathcal N((n,n+n^\theta])}{\Lambda_{\alpha,\beta}((n,n+n^\theta])}-1\right|>\epsilon\right)<\infty,$$
	and hence, due to the Borel-Cantelli lemma, it holds that
	$$ \mathcal N((x,x+x^\theta])\sim\Lambda_{\alpha,\beta}((x,x+x^\theta]),~\mbox{almost surely, as}~x\rightarrow\infty.$$
	Since $\Lambda_{\alpha,\beta}((x,x+x^\theta])\sim\frac{\alpha x^\theta}{\log x}$, it follows that
	\begin{equation}
		\mathcal N((x,x+x^\theta])\sim\alpha\frac{x^\theta}{\log x},~\mbox{almost surely, as}~x\rightarrow\infty.
		\label{eq:pgap3}
	\end{equation}
	Since (\ref{eq:pgap3}) holds given any $\alpha>0$, (\ref{eq:embedding1}) holds when $\alpha=1$.
\end{proof}

\begin{remark}
	\label{remark:remark_pgp}
	Under assumptions on the Riemann zeta function, \ctn{Ing1937} had shown that, as $x\rightarrow\infty$,
	$$\varphi(x+x^\theta)-\varphi(x)\sim\frac{x^\theta}{\log x},~\mbox{for any}~\theta>\frac{1+4c}{2+4c},$$
	for some $c>0$. This is in keeping with Theorem \ref{theorem:pgap1}, but importantly, our result does not require
	any assumption on the Rienann zeta function.
\end{remark}

\begin{remark}
	\label{remark:pgap2}
	It is easy to see that Theorem \ref{theorem:pgap1} remains valid for any error function $f(\cdot)$ satisfying
	$$\frac{\hat\beta_x\int_x^{x+x^\theta}f(u)du}{x^\theta/\log x}\rightarrow 0,~\mbox{as}~x\rightarrow\infty.$$
\end{remark}

\section{Recursive Bayesian theory for prime prediction: the key ideas}
\label{sec:recursive}

\subsection{Recursive posteriors of the parameters}
\label{subsec:recursive_posteriors}
Our first goal is to learn the parameters recursively. That is,
given $Z_1=t_1$, one can learn $(\alpha,\beta)$ from $f_{Z_1,\alpha,\beta}(t_1)$; let the learned value be denoted by 
$(\alpha_1,\beta_1)$. Then the resultant density of $Z_1$ at $t_1$ becomes $f_{Z_1,\alpha_1,\beta_1}(t_1)$. 
Now, conditional on $Z_1=t_1$, given $Z_2=t_2$, $(\alpha,\beta)$ can be learned from the density $f_{Z_2,\alpha,\beta}(t_2)$. Letting the learned value be $(\alpha_2,\beta_2)$,
the density of $Z_2$ at $t_2$ becomes $f_{Z_2,\alpha_2,\beta_2}(t_2|Z_1=t_1)$.
Thus, in general, for $k\geq 2$, we shall denote the $k$-th conditional density by $f_{Z_k,\alpha_k,\beta_k}(t_k|Z_{k-1}=t_{k-1})$.

Let $$B^{(0)}_1=\int_2^{t_1}li(u)du~~\mbox{and for}~k\geq 1,~\mbox{let}~ B^{(k)}_1=\int_{t_k}^{t_{k+1}}li(u)du.$$ Also, let 
$$B^{(0)}_2=\int_2^{t_1}f(u)du=~\mbox{and for}~k\geq 1,~~B^{(k)}_2=\int_{t_k}^{t_{k+1}}f(u)du.$$ Further let
$$C^{(0)}_1=li(t_1)~\mbox{and for}~t\geq 1,~ C^{(k)}_1=li(t_{k+1}).$$ Finally, let
$$C^{(0)}_2=f(t_1)~\mbox{and for}~t\geq 1,~C^{(k)}_2=f(t_{k+1}).$$ 

Then note that
\begin{equation}
f_{Z_1,\alpha_1,\beta_1}(t_1)=\exp\left\{-\left(\alpha_1 B^{(0)}_1+\beta_1 B^{(0)}_2\right)\right\}\left(\alpha_1 C^{(0)}_1+\beta_1 C^{(0)}_2\right)
	\label{eq:z1_dens}
\end{equation}
and for $k\geq 1$,
\begin{equation}
	f_{Z_{k+1},\alpha_{k+1},\beta_{k+1}}(t_{k+1}|Z_k=t_k)=\exp\left\{-\left(\alpha_{k+1} B^{(k)}_1+\beta_{k+1} B^{(k)}_2\right)\right\}
	\left(\alpha_{k+1} C^{(k)}_1+\beta_{k+1} C^{(k)}_2\right)
	\label{eq:zk_dens}
\end{equation}

We learn about the parameters in the following recursive Bayesian way. With some prior 
$\pi(\alpha_1,\beta_1)$ on $(\alpha_1,\beta_1)$, we first compute the posterior distribution of $(\alpha_1,\beta_1)$ given $Z_1=t_1$, that is,
we obtain the posterior $\pi(\alpha_1,\beta_1|Z_1=t_1)\propto\pi(\alpha_1,\beta_1)f_{Z_1,\alpha_1,\beta_1}(t_1)$. To learn about $(\alpha_2,\beta_2)$, 
we first essentially set the prior of $(\alpha_2,\beta_2)$ to be the posterior $\pi(\alpha_1,\beta_1|Z_1=t_1)$, with $(\alpha_1,\beta_1)$ replaced by $(\alpha_2,\beta_2)$.
Then we obtain the posterior of $(\alpha_2,\beta_2)$ given $Z_2=t_2$ essentially as 
$\pi(\alpha_2,\beta_2|Z_2=t_2)\propto\pi(\alpha_2,\beta_2|Z_1=t_1)f_{Z_2=t_2,\alpha_2,\beta_2)}(t_2|Z_1=t_1)$.
In general, for $k\geq 2$, we learn about $(\alpha_k,\beta_k)$ by essentially computing the posterior 
$\pi(\alpha_k,\beta_k|Z_k=t_k)\propto\pi(\alpha_k,\beta_k|Z_{k-1}=t_{k-1})f_{Z_k,\alpha_k,\beta_k}(t_k|Z_{k-1}=t_{k-1})$.

\subsection{Prediction with recursive posteriors}
\label{subsec:recursive_prediction}
The posterior predictive density of $Z_{k+1}$ at any point $t$ is given by
\begin{equation}
	\int_0^{\infty}\int_0^{\infty}f_{Z_{k+1},\alpha_{k},\beta_{k}}(t|Z_k=t_k)\pi(\alpha_{k},\beta_{k}|Z_k=t_k)d\alpha_{k} d\beta_{k},\label{eq:prediction}
\end{equation}
where we have notationally replaced $f_{Z_{k+1},\alpha_{k+1},\beta_{k+1}}(t|Z_k=t_k)$ with $f_{Z_{k+1},\alpha_{k},\beta_{k}}(t|Z_k=t_k)$. The reason for this
is that, since $Z_{k+1}$ is not observed, $\alpha_{k+1},\beta_{k+1})$ must be learned from the prior, and here the prior for $(\alpha_{k+1},\beta_{k+1})$
is nothing but the available posterior $\pi(\alpha_{k},\beta_{k}|Z_k=t_k)$, subject to the only requirement of notational adjustment for the parameter indices.

\section{Recursive Bayesian theory: the details}
\label{sec:recursive_details}

\subsection{Recursive posteriors}
\label{subsec:recursive_posteriors2}
Let the prior for $(\alpha_1,\beta_1)$ be given by
\begin{equation}
	\pi(\alpha_1,\beta_1)\propto\exp\left(-a\alpha_1\right)\alpha^{\gamma-1}_1\times\exp\left(-b\beta_1\right)\beta^{\xi-1}_1.
	\label{eq:prior1}
\end{equation}
The posterior of $(\alpha_1,\beta_1)$ given $Z_1=t_1$ is then 
\begin{align}
	&\pi(\alpha_1,\beta_1|Z_1=t_1)\propto\pi(\alpha_1,\beta_1)f_{Z_1,\alpha_1,\beta_1}(t_1)\notag\\
	&\qquad=
	C^{(0)}_1\exp\left\{-\alpha_1\left(a+B^{(0)}_1\right)\right\} \exp\left\{-\beta_1\left(b+B^{(0)}_2\right)\right\}\alpha^{\gamma+1-1}_1\beta^{\xi-1}_1\notag\\
	&\qquad\qquad+
	C^{(0)}_2\exp\left\{-\alpha_1\left(a+B^{(0)}_1\right)\right\} \exp\left\{-\beta_1\left(b+B^{(0)}_2\right)\right\}\alpha^{\gamma-1}_1\beta^{\xi+1-1}_1,\label{eq:gamma_mix1}
\end{align}
which is a mixture of two products of gamma densities.

For the prior of $(\alpha_2,\beta_2)$, we slightly modify the posterior (\ref{eq:gamma_mix1}) to set
\begin{equation}
	\pi(\alpha_2,\beta_2)\propto\exp\left\{-\alpha_1\left(a+B^{(0)}_1\right)\right\} \exp\left\{-\beta_1\left(b+B^{(0)}_2\right)\right\}\alpha^{\gamma+1-1}_1\beta^{\xi+1-1}_1,
	\label{eq:gamma_mix2}
\end{equation}
which is obtained by adding a power to $\beta_1$ in the first mixture component and adding a power to $\alpha_1$ in the second mixture component. With this prior
we obtain the following posterior of $(\alpha_2,\beta_2)$ given $Z_2=t_2$:
\begin{align}
	&\pi(\alpha_2,\beta_2|Z_2=t_2)\propto\pi(\alpha_2,\beta_2)f_{Z_2,\alpha_2,\beta_2}(t_2)\notag\\
	&\qquad=
	C^{(1)}_1\exp\left\{-\alpha_1\left(a+B^{(0)}_1+B^{(1)}_1\right)\right\} \exp\left\{-\beta_1\left(b+B^{(0)}_2+B^{(1)}_2\right)\right\}\alpha^{\gamma+2-1}_1\beta^{\xi+1-1}_1\notag\\
	&\qquad\qquad+
	C^{(1)}_2\exp\left\{-\alpha_1\left(a+B^{(0)}_1+B^{(1)}_1\right)\right\} \exp\left\{-\beta_1\left(b+B^{(0)}_2+B^{(1)}_2\right)\right\}\alpha^{\gamma+1-1}_1\beta^{\xi+2-1}_1,\label{eq:gamma_mix3}
\end{align}
Thus, in general, for $k\geq 2$, we have
\begin{equation}
	\pi(\alpha_k,\beta_k)\propto\exp\left\{-\alpha_1\left(a+\sum_{i=0}^{k-2}B^{(i)}_1\right)\right\} \exp\left\{-\beta_1\left(b+\sum_{i=0}^{k-2}B^{(i)}_2\right)\right\}
	\alpha^{\gamma+k-1-1}_1\beta^{\xi+k-1-1}_1,
	\label{eq:gamma_mix4}
\end{equation}
and
\begin{align}
	&\pi(\alpha_k,\beta_k|Z_k=t_k)\propto\pi(\alpha_k,\beta_k)f_{Z_k,\alpha_k,\beta_k}(t_k)\notag\\
	&\qquad=
	C^{(k-1)}_1\exp\left\{-\alpha_k\left(a+\sum_{i=0}^{k-1}B^{(i)}_1\right)\right\} \exp\left\{-\beta_k\left(b+\sum_{i=0}^{k-1}B^{(i)}_2\right)\right\}
	\alpha^{\gamma+k-1}_k\beta^{\xi+k-1-1}_k\notag\\
	&\qquad\qquad+
	C^{(k-1)}_2\exp\left\{-\alpha_k\left(a+\sum_{i=0}^{k-1}B^{(i)}_1\right)\right\} \exp\left\{-\beta_k\left(b+\sum_{i=0}^{k-1}B^{(i)}_2\right)\right\}
	\alpha^{\gamma+k-1-1}_k\beta^{\xi+k-1}_k.\label{eq:gamma_mix5}
\end{align}

For any $a>0,b>0$ let $g(x:a,b)$ denote the gamma density given by $g(x:a,b)=\frac{a^b}{\Gamma(b)}\exp(-ax)x^{b-1}$.
Then the closed form expression for the posterior $\pi(\alpha_k,\beta_k|Z_k=t_k)$ is given by
\begin{align}
	&\pi(\alpha_k,\beta_k|Z_k=t_k)\notag\\
	&\qquad=R_kC^{(k-1)}_1\times
	\frac{\Gamma(\gamma+k)}{\left(a+\sum_{i=0}^{k-1}B^{(i)}_1\right)^{\gamma+k}} 
	\times\frac{\Gamma(\xi+k-1)}{\left(b+\sum_{i=0}^{k-1}B^{(i)}_2\right)^{\xi+k-1}}\notag\\
	&\qquad\qquad\times g\left(\alpha_k:a+\sum_{i=0}^{k-1}B^{(i)}_1,\gamma+k\right)
	\times g\left(\beta_k:b+\sum_{i=0}^{k-1}B^{(i)}_2,\xi+k-1\right)\notag\\
	&\qquad+R_kC^{(k-1)}_2\times
	\frac{\Gamma(\gamma+k-1)}{\left(a+\sum_{i=0}^{k-1}B^{(i)}_1\right)^{\gamma+k-1}} 
	\times\frac{\Gamma(\xi+k)}{\left(b+\sum_{i=0}^{k-1}B^{(i)}_2\right)^{\xi+k}}\notag\\
	&\qquad\qquad\times g\left(\alpha_k:a+\sum_{i=0}^{k-1}B^{(i)}_1,\gamma+k-1\right)
	\times g\left(\beta_k:b+\sum_{i=0}^{k-1}B^{(i)}_2,\xi+k\right),\label{eq:post_k}
\end{align}
where $R_k$ is the normalizing constant, given by
\begin{align}
	&R^{-1}_k
	=C^{(k-1)}_1\times
	\frac{\Gamma(\gamma+k)}{\left(a+\sum_{i=0}^{k-1}B^{(i)}_1\right)^{\gamma+k}} 
	\times\frac{\Gamma(\xi+k-1)}{\left(b+\sum_{i=0}^{k-1}B^{(i)}_2\right)^{\xi+k-1}}\notag\\
	&\qquad+C^{(k-1)}_2\times
	\frac{\Gamma(\gamma+k-1)}{\left(a+\sum_{i=0}^{k-1}B^{(i)}_1\right)^{\gamma+k-1}} 
	\times\frac{\Gamma(\xi+k)}{\left(b+\sum_{i=0}^{k-1}B^{(i)}_2\right)^{\xi+k}}.\label{eq:normconst_k}
\end{align}

\subsection{Recursive prediction}
\label{subsec:recursive_prediction2}
Using (\ref{eq:prediction}), we obtain
the posterior predictive density of $Z_{k+1}$ at any point $t=t_{k+1}$, where $t_{k+1}>t_k$, to be
\begin{align}
&\pi(Z_{k+1}=t_{k+1}|Z_k=t_k)=	\int_0^{\infty}\int_0^{\infty}f_{Z_{k+1},\alpha_{k},\beta_{k}}(t=t_{k+1}|Z_k=t_k)\pi(\alpha_{k},\beta_{k}|Z_k=t_k)d\alpha_{k} d\beta_{k}\notag\\
	&\qquad
	=R_kC^{(k)}_1C^{(k-1)}_1\times
	\frac{\Gamma(\gamma+k+1)}{\left(a+\sum_{i=0}^{k}B^{(i)}_1\right)^{\gamma+k+1}} 
	\times\frac{\Gamma(\xi+k-1)}{\left(b+\sum_{i=0}^{k}B^{(i)}_2\right)^{\xi+k-1}}\notag\\
	&\qquad+R_kC^{(k)}_2C^{(k-1)}_1\times
	\frac{\Gamma(\gamma+k)}{\left(a+\sum_{i=0}^{k}B^{(i)}_1\right)^{\gamma+k}} 
	\times\frac{\Gamma(\xi+k)}{\left(b+\sum_{i=0}^{k}B^{(i)}_2\right)^{\xi+k}}\notag\\
	&\qquad+R_kC^{(k)}_1C^{(k-1)}_2\times
	\frac{\Gamma(\gamma+k)}{\left(a+\sum_{i=0}^{k}B^{(i)}_1\right)^{\gamma+k}} 
	\times\frac{\Gamma(\xi+k)}{\left(b+\sum_{i=0}^{k}B^{(i)}_2\right)^{\xi+k}}\notag\\
	&\qquad+R_kC^{(k)}_2C^{(k-1)}_2\times
	\frac{\Gamma(\gamma+k-1)}{\left(a+\sum_{i=0}^{k}B^{(i)}_1\right)^{\gamma+k-1}} 
	\times\frac{\Gamma(\xi+k+1)}{\left(b+\sum_{i=0}^{k}B^{(i)}_2\right)^{\xi+k+1}}.
	\label{eq:post_pred_k}
\end{align}

\section{Convergence of the marginal posteriors of the parameters and refutation of RH}
\label{sec:con}

In the remaining of this article, we shall denote ``converges in probability" by $\stackrel{P}{\longrightarrow}$.
\subsection{Posterior convergence of $\alpha_k$ and validation of the $Li$ function}
\label{subsec:alpha_conv}

\begin{theorem}
	\label{theorem:pconv_alpha}
	Under the marginal posterior probability measures $\pi(\alpha_k|Z_k=t_k)$; $k\geq 1$,
	\begin{equation}
		\alpha_k\stackrel{P}{\longrightarrow} 1,~\mbox{as}~k\rightarrow\infty.
		\label{eq:pconv_alpha}
	\end{equation}
\end{theorem}
\begin{proof}
From (\ref{eq:post_k}) it follows that the marginal posterior of $\alpha_k$ is
\begin{align}
	&\pi(\alpha_k|Z_k=t_k)\notag\\
	&\qquad=R_kC^{(k-1)}_1\times
	\frac{\Gamma(\gamma+k)}{\left(a+\sum_{i=0}^{k-1}B^{(i)}_1\right)^{\gamma+k}} 
	\times\frac{\Gamma(\xi+k-1)}{\left(b+\sum_{i=0}^{k-1}B^{(i)}_2\right)^{\xi+k-1}}\notag\\
	&\qquad\qquad\times g\left(\alpha_k:a+\sum_{i=0}^{k-1}B^{(i)}_1,\gamma+k\right)\notag\\
	&\qquad+R_kC^{(k-1)}_2\times
	\frac{\Gamma(\gamma+k-1)}{\left(a+\sum_{i=0}^{k-1}B^{(i)}_1\right)^{\gamma+k-1}} 
	\times\frac{\Gamma(\xi+k)}{\left(b+\sum_{i=0}^{k-1}B^{(i)}_2\right)^{\xi+k}}\notag\\
	&\qquad\qquad\times g\left(\alpha_k:a+\sum_{i=0}^{k-1}B^{(i)}_1,\gamma+k-1\right).
	\label{eq:post_alpha_k}
\end{align}
Then
\begin{align}
	&E(\alpha_k|Z_k=t_k)\notag\\
	&\qquad=R_kC^{(k-1)}_1\times
	\frac{\Gamma(\gamma+k)}{\left(a+\sum_{i=0}^{k-1}B^{(i)}_1\right)^{\gamma+k}} 
	\times\frac{\Gamma(\xi+k-1)}{\left(b+\sum_{i=0}^{k-1}B^{(i)}_2\right)^{\xi+k-1}}\notag\\
	&\qquad\qquad\times \frac{\gamma+k}{a+\sum_{i=0}^{k-1}B^{(i)}_1}\notag\\
	&\qquad+R_kC^{(k-1)}_2\times
	\frac{\Gamma(\gamma+k-1)}{\left(a+\sum_{i=0}^{k-1}B^{(i)}_1\right)^{\gamma+k-1}} 
	\times\frac{\Gamma(\xi+k)}{\left(b+\sum_{i=0}^{k-1}B^{(i)}_2\right)^{\xi+k}}\notag\\
	&\qquad\qquad\times \frac{\gamma+k-1}{a+\sum_{i=0}^{k-1}B^{(i)}_1}.
	\label{eq:postmean_alpha_k}
\end{align}
Now note that
\begin{equation}
	\sum_{i=0}^{k-1}B^{(i)}_1=\int_2^{t_k}li(u)du\sim k~\mbox{as}~k\rightarrow\infty,
	\label{eq:no_primes}
\end{equation}
	since there are $k$ primes $t_1,\ldots,t_k$ in the interval $[2,t_k]$ considered in the above log-integral.
It follows that
\begin{equation}
 \frac{\gamma+k}{a+\sum_{i=0}^{k-1}B^{(i)}_1}\rightarrow 1~\mbox{and}~ \frac{\gamma+k-1}{a+\sum_{i=0}^{k-1}B^{(i)}_1}\rightarrow 1,~\mbox{as}~k\rightarrow\infty.
	\label{eq:lim1}
\end{equation}
In other words, for any $\epsilon>0$, there exists $k_0\geq 1$ such that for $k\geq k_0$,
$$ 1-\epsilon<\frac{\gamma+k}{a+\sum_{i=0}^{k-1}B^{(i)}_1}< 1+\epsilon~\mbox{and}~ 1-\epsilon<\frac{\gamma+k-1}{a+\sum_{i=0}^{k-1}B^{(i)}_1}<1+\epsilon.$$
Applying to (\ref{eq:lim1}) to (\ref{eq:postmean_alpha_k}) we obtain
\begin{equation*}
	1-\epsilon<E(\alpha_k|Z_k=t_k)<1+\epsilon,
\end{equation*}
that is,
\begin{equation}
	E(\alpha_k|Z_k=t_k)\rightarrow 1,~\mbox{as}~k\rightarrow\infty.
	\label{eq:meanlim1}
\end{equation}
Noting that the second moment of $g(x:a,b)$ is $\frac{a(a+1)}{b^2}$, for $a,b>0$, we similarly obtain
\begin{equation}
	E(\alpha^2_k|Z_k=t_k)\rightarrow 1,~\mbox{as}~k\rightarrow\infty.
	\label{eq:meansqlim1}
\end{equation}
Combining (\ref{eq:meanlim1}) and (\ref{eq:meansqlim1}) yields
\begin{equation}
	Var(\alpha_k|Z_k=t_k)\rightarrow 0,~\mbox{as}~k\rightarrow\infty.
	\label{eq:varlim1}
\end{equation}
Combining (\ref{eq:meanlim1}) and (\ref{eq:varlim1}) proves (\ref{eq:pconv_alpha}).
\end{proof}
\begin{remark}
	\label{remark:remark_Li}
	Theorem \ref{theorem:pconv_alpha} provides a validation of the $Li$ function in the PNT, since its coefficient, $\alpha_k$, converges to $1$ in probability.
	This is of course, just a re-validation of the already well-established PNT.
\end{remark}

\subsection{Posterior convergence of $\beta_k$ and falsification of RH}
\label{subsec:postconv__beta}
\begin{theorem}
	\label{theorem:pconv_beta}
	Under the marginal posterior probability measures $\pi(\beta_k|Z_k=t_k)$; $k\geq 1$,
	\begin{equation}
		\beta_k-E(\beta_k|Z_k=t_k)\stackrel{P}{\longrightarrow} 0,~\mbox{as}~k\rightarrow\infty,
		\label{eq:pconv_beta}
	\end{equation}
	where $E(\beta_k|Z_k=t_k)\sim\frac{\sqrt{k}}{\left(\log k\right)^{3/2}}$, as $k\rightarrow\infty$.
\end{theorem}
\begin{proof}
From (\ref{eq:post_k}) it follows that the marginal posterior of $\alpha_k$ is
\begin{align}
	&\pi(\beta_k|Z_k=t_k)\notag\\
	&\qquad=R_kC^{(k-1)}_1\times
	\frac{\Gamma(\gamma+k)}{\left(a+\sum_{i=0}^{k-1}B^{(i)}_1\right)^{\gamma+k}} 
	\times\frac{\Gamma(\xi+k-1)}{\left(b+\sum_{i=0}^{k-1}B^{(i)}_2\right)^{\xi+k-1}}\notag\\
	&\qquad\qquad\times g\left(\beta_k:b+\sum_{i=0}^{k-1}B^{(i)}_2,\xi+k-1\right)\notag\\
	&\qquad+R_kC^{(k-1)}_2\times
	\frac{\Gamma(\gamma+k-1)}{\left(a+\sum_{i=0}^{k-1}B^{(i)}_1\right)^{\gamma+k-1}} 
	\times\frac{\Gamma(\xi+k)}{\left(b+\sum_{i=0}^{k-1}B^{(i)}_2\right)^{\xi+k}}\notag\\
	&\qquad\qquad\times g\left(\beta_k:b+\sum_{i=0}^{k-1}B^{(i)}_2,\xi+k\right).
	\label{eq:post_beta_k}
\end{align}
Then
\begin{align}
	&E(\beta_k|Z_k=t_k)\notag\\
	&\qquad=R_kC^{(k-1)}_1\times
	\frac{\Gamma(\gamma+k)}{\left(a+\sum_{i=0}^{k-1}B^{(i)}_1\right)^{\gamma+k}} 
	\times\frac{\Gamma(\xi+k-1)}{\left(b+\sum_{i=0}^{k-1}B^{(i)}_2\right)^{\xi+k-1}}\notag\\
	&\qquad\qquad\times \frac{\xi+k-1}{b+\sum_{i=0}^{k-1}B^{(i)}_2}\notag\\
	&\qquad+R_kC^{(k-1)}_2\times
	\frac{\Gamma(\gamma+k-1)}{\left(a+\sum_{i=0}^{k-1}B^{(i)}_1\right)^{\gamma+k-1}} 
	\times\frac{\Gamma(\xi+k)}{\left(b+\sum_{i=0}^{k-1}B^{(i)}_2\right)^{\xi+k}}\notag\\
	&\qquad\qquad\times \frac{\xi+k}{b+\sum_{i=0}^{k-1}B^{(i)}_2}.
	\label{eq:postmean_beta_k}
\end{align}
Here note that $\sum_{i=0}^{k-1}B^{(i)}_2=\sqrt{t_k}\log(t_k)$.
Now, combining the result $\int_2^{t_k}li(u)du\sim\frac{t_k}{\log(t_k)}$ with (\ref{eq:no_primes}), we have
$\frac{t_k}{\log(t_k)}\sim k$, so that $\sqrt{t_k}\log(t_k)\sim \frac{t^{3/2}_k}{k}$. In other words,
\begin{equation}
	\sum_{i=0}^{k-1}B^{(i)}_2\sim \frac{t^{3/2}_k}{k},~\mbox{as}~k\rightarrow\infty.
	\label{eq:beta_approx1}
\end{equation}
	Using (\ref{eq:beta_approx1}) it is easily seen that
	\begin{equation}
		\mu^{(1)}_k=\frac{\xi+k-1}{b+\sum_{i=0}^{k-1}B^{(i)}_2}\sim \frac{k^2}{t^{3/2}_k}~\mbox{and}~
		 \mu^{(2)}_k=\frac{\xi+k}{b+\sum_{i=0}^{k-1}B^{(i)}_2}\sim \frac{k^2}{t^{3/2}_k},~\mbox{as}~k\rightarrow\infty.
		 \label{eq:beta_approx2}
	\end{equation}
	Hence, for any $\epsilon>0$, there exists $k_0\geq 1$ such that for $k\geq k_0$,
	\begin{equation}
		(1-\epsilon)\frac{k^2}{t^{3/2}_k}<\mu^{(i)}_k<(1+\epsilon)\frac{k^2}{t^{3/2}_k},~\mbox{for}~i=1,2.
		\label{eq:beta_approx3}
	\end{equation}
	It follows from (\ref{eq:beta_approx3}) that
	\begin{equation*}
		(1-\epsilon)\frac{k^2}{t^{3/2}_k}<E(\beta_k|Z_k=t_k)<(1+\epsilon)\frac{k^2}{t^{3/2}_k},
	\end{equation*}
	that is,
	\begin{equation}
		E(\beta_k|Z_k=t_k)\sim \frac{k^2}{t^{3/2}_k},~\mbox{as}~k\rightarrow\infty.
		\label{eq:beta_approx4}
	\end{equation}
	Since $t_k\sim k\log k$ as $k\rightarrow\infty$, it follows that 
	\begin{equation}
		\frac{t^{3/2}_k}{k}\sim\sqrt{k}\left(\log k\right)^{3/2}\rightarrow\infty~\mbox{and}~
		\frac{k^2}{t^{3/2}_k}\sim\frac{\sqrt{k}}{\left(\log k\right)^{3/2}}\rightarrow\infty,~\mbox{as}~k\rightarrow\infty.\label{eq:beta_approx5}
	\end{equation}
Now note that, under $g\left(\alpha_k:b+\sum_{i=0}^{k-1}B^{(i)}_2,\xi+k-1\right)$ and $g\left(\beta_k:b+\sum_{i=0}^{k-1}B^{(i)}_2,\xi+k\right)$,
the second order moments are given by
\begin{equation}
	\mu^{(1)}_{2k}=\frac{(\xi+k-1)(\xi+k)}{\left(b+\sum_{i=0}^{k-1}B^{(i)}_2\right)^2}~\mbox{and}~
	\mu^{(2)}_{2k}=\frac{(\xi+k)(\xi+k+1)}{\left(b+\sum_{i=0}^{k-1}B^{(i)}_2\right)^2}.
	\label{eq:beta_approx6}
\end{equation}
Using (\ref{eq:beta_approx6}) and noting that $\mu^{(1)}_k<\mu^{(2)}_k$, we obtain
\begin{equation}
	E\left(\beta^2_k|Z_k=t_k\right)<\mu^{(2)}_{2k}.
	\label{eq:beta_approx7}
\end{equation}
Also, it follows from (\ref{eq:beta_approx2}) that $\mu^{(2)}_k>\mu^{(1)}_k$; using this in (\ref{eq:postmean_beta_k}) yields
\begin{equation}
	E\left(\beta_k|Z_k=t_k\right)>\mu^{(1)}_{k}.
	\label{eq:beta_approx8}
\end{equation}
Combining (\ref{eq:beta_approx7}) and (\ref{eq:beta_approx8}) and using (\ref{eq:beta_approx1}) along with $t_k\sim k\log k$ gives
\begin{align}
	Var(\beta_k|Z_k=t_k)&<\mu^{(2)}_{2k}-\left(\mu^{(1)}_k\right)^2
	=\frac{3(\xi+k)-1}{\left(b+\sum_{i=0}^{k-1}B^{(i)}_2\right)^2}\notag\\
	&\sim\frac{3(\xi+k)-1}{t^3_k/k^2}
	\sim\frac{k^2\left[3(\xi+k)-1\right]}{k^3\left(\log k\right)^3}\notag\\
	&\rightarrow 0,~\mbox{as}~k\rightarrow\infty.
	\label{eq:postvar_beta_k}
\end{align}
Let $\mu_k=E(\beta_k|Z_k=t_k)$. Then from (\ref{eq:beta_approx4}) and (\ref{eq:beta_approx5}) it follows that 
\begin{equation}
	\mu_k=E(\beta_k|Z_k=t_k)\rightarrow\infty,~\mbox{as}~k\rightarrow\infty.
	\label{eq:beta_approx9}
\end{equation}
However, it follows from $E(\beta_k-\mu_k|Z_k=t_k)=0$ and (\ref{eq:postvar_beta_k}), that (\ref{eq:pconv_beta}) holds.
\end{proof}

\begin{theorem}
	\label{theorem:n_bound}
	With respect to the posterior distributions $\pi(\beta_k|Z_k=t_k)$; $k\geq 1$, there does not exist any constant $0< M<\infty$ such that
	$\beta_k<M$ almost surely or in probability, as $k\rightarrow\infty$.
\end{theorem}
\begin{proof}
	By Theorem \ref{theorem:pconv_beta}, $\beta_k-E(\beta_k|Z_k=t_k)$ converges to zero in probability. Hence, there exists a
	subsequence $\{k_j\}_{j=1}^{\infty}$ such that $\left\{\beta_{k_j}-E(\beta_{k_j}|Z_{k_j}=t_{k_j})\right\}_{j=1}^{\infty}$ converges to zero almost surely;
	see, for example,\ctn{Schervish95}.
	That is, given any $\epsilon>0$, there exists $j_0\geq 1$ such that for $j\geq j_0$, $\left|\beta_{k_j}-E(\beta_{k_j}|Z_{k_j}=t_{k_j})\right|<\epsilon$, 
	outside the null set, 
	Hence, $\beta_{k_j}>E(\beta_{k_j}|Z_{k_j}=t_{k_j})-\epsilon$, for $j\geq j_0$. Since $E(\beta_{k_j}|Z_{k_j}=t_{k_j})\rightarrow\infty$ as $j\rightarrow\infty$, 
	there can not exist any finite constant $M>0$ such that $\beta_k<M$ as $k\rightarrow\infty$, outside the null set. 
	
	For the ``in probability" part, note that since for any $\epsilon>0$, $\pi(E(\beta_k|Z_k=t_k)-\epsilon<\beta_k<E(\beta_k|Z_k=t_k)+\epsilon|Z_k=t_k)\rightarrow 1$
	as $k\rightarrow\infty$, it is clear that for any finite $M>0$, $\pi(\beta_k<M|Z_k=t_k)\rightarrow 0$, as $k\rightarrow\infty$.
\end{proof}

\begin{corollary}
	\label{cor:cor1}
With respect to our Bayesian theory, RH is false.
\end{corollary}

\section{Another evaluation of RH: falsification again}
\label{sec:another}
RH is equivalent to the statement (\ctn{titchmarsh1986zeta})
\begin{equation}
	\varphi(x)=Li(x)+O\left(x^{1/2+\epsilon}\right),~\mbox{for all}~\epsilon>0.
	\label{eq:rh2}
\end{equation}
To evaluate this statement, we now consider $$\Lambda_{\alpha,\beta}(x)=\int_A li(u)du+\int_A f(u)du,$$ where
\begin{equation}
f(x)=\left(\frac{1}{2}+\epsilon\right)x^{\epsilon-1/2}
	\label{eq:rh3}
\end{equation}
is the derivative of the error function $x^{1/2+\epsilon}$.

\begin{theorem}
	\label{theorem:conv_rh}
	Consider an NHPP with intensity function
\begin{equation}
	\Lambda_{\alpha,\beta}(A)=\alpha\int_A li(u)du+\beta\int_A f(u)du,
	\label{eq:intensity_measure_rh}
\end{equation}
	where the function $f(\cdot)$ is given by (\ref{eq:rh3}). 
	Let $\alpha_k$ and $\beta_k$ denote the stage-wise parameters corresponding to $\alpha$ and $\beta$ in (\ref{eq:intensity_measure_rh}).
	Then, with respect to the respective posteriors $\pi(\alpha_k|Z_k=t_k)$
	and $\pi(\beta_k|Z_k=t_k)$; $k\geq 1$, the following hold:
	\begin{equation}
		\alpha_k\stackrel{P}{\longrightarrow}1~\mbox{and}~\beta_k-E(\beta_k|Z_k=t_k)\stackrel{P}{\longrightarrow 0},~\mbox{as}~k\rightarrow\infty.
		\label{eq:conv_rh1}
	\end{equation}
	where
	\begin{equation}
		E(\beta_k|Z_k=t_k)\sim  \frac{k^{1/2-\epsilon}}{\left(\log k\right)^{1/2+\epsilon}}\rightarrow\infty,~\mbox{as}~k\rightarrow\infty.
		\label{eq:conv_rh2}
	\end{equation}
\end{theorem}
\begin{proof}
	The proof follows in the same way as the proofs of  Theorems \ref{theorem:pconv_alpha} and \ref{theorem:pconv_beta}, by noting that here
	$\sum_{i=0}^{k-1}B^{(i)}_2=t^{1/2+\epsilon}_k$ and that 
	$$E(\beta_k|Z_k=t_k)=\frac{k}{t^{1/2+\epsilon}_k}\sim\frac{k^{1/2-\epsilon}}{\left(\log k\right)^{1/2+\epsilon}},~\mbox{as}~k\rightarrow\infty.$$ 
	Also note that the form of $C^{(k-1)}_2$ here is different from that of Theorem \ref{theorem:pconv_beta},
	but the exact form is unimportant for the proof. 
\end{proof}
\begin{corollary}
	\label{cor:cor_rh}
	When $0<\epsilon<1/2$,
	$$E(\beta_k|Z_k=t_k)\rightarrow\infty,~\mbox{as}~k\rightarrow\infty,$$
	and hence, the conclusion of Theorem \ref{theorem:n_bound} holds once again, providing strong evidence against RH. 
	In other words, we again conclude that with respect to our Bayesian theory, RH is false.
\end{corollary}

\section{Construction of an upper bound for the prime counting function with our Bayesian theory}
\label{eq:upper_bound}
By Corollary \ref{cor:cor1}, RH fails to be true, at least with respect to our Bayesian theory. The error term $O(\sqrt{x}\log x)$ turns out to be
too small to be true. As such, we now attempt to furnish an appropriate error term.

Since $E(\beta_k|Z_k=t_k)\sim\frac{k^2}{t^{3/2}_k}\sim\frac{\sqrt{k}}{\left(\log k\right)^{3/2}}\rightarrow\infty$ as $k\rightarrow\infty$, it makes 
sense to multiply this asymptotic form with the existing error term $\sqrt{t_k}\log t_k\sim \sqrt{k}\log k$ to obtain $\frac{k}{\sqrt{\log k}}$, that is,
$\int_2^xf(u) du=\frac{x}{\sqrt{\log x}}$ in (\ref{eq:intensity_measure}).
However, this modification yields $E(\beta_k|Z_k=t_k)\rightarrow 0$ and $Var(\beta_k|Z_k=t_k)\rightarrow 0$, as $k\rightarrow\infty$, that is,
$\beta_k\stackrel{P}{\longrightarrow}0$, with respect to the posteriors of $\beta_k$; $k\geq 1$. In other words, this modified error terms turns out to be too large.

However, instead of multiplying the current error term by $\frac{\sqrt{k}}{\left(\log k\right)^{3/2}}$, if it is multiplied by 
$\frac{\sqrt{k}}{\left(\log k\right)^{2}}$, the power of the denominator being the ceiling of that of the former, then the new error term becomes
$\frac{x}{\log x}$. This yields desired convergence of the parameters formalized by the theorem below.

\begin{theorem}
	\label{theorem:conv2}
	Consider an NHPP with intensity function
\begin{equation}
	\Lambda_{\alpha,\beta}(A)=\alpha\int_A li(u)du+\beta\int_A f(u)du,
	\label{eq:intensity_measure2}
\end{equation}
	where the function $f(\cdot)$ is such that $$\int_2^xf(u)du=\frac{x}{\log x}.$$ 
	Let $\alpha_k$ and $\beta_k$ denote the stage-wise parameters corresponding to $\alpha$ and $\beta$ in (\ref{eq:intensity_measure2}).
	Then, with respect to the respective posteriors $\pi(\alpha_k|Z_k=t_k)$
	and $\pi(\beta_k|Z_k=t_k)$; $k\geq 1$, the following hold:
	\begin{equation}
		\alpha_k\stackrel{P}{\longrightarrow}1~\mbox{and}~\beta_k\stackrel{P}{\longrightarrow 1},~\mbox{as}~k\rightarrow\infty.
		\label{eq:conv2}
	\end{equation}
\end{theorem}
\begin{proof}
	Noting that here $\sum_{i=0}^{k-1}B^{(i)}_2=\frac{t_k}{\log t_k}$, the proof follows in the same way as before by showing that 
	$E(\alpha_k|Z_k=t_k)\rightarrow 1$, $Var(\alpha_k|Z_k=t_k)\rightarrow 0$ and
	$E(\beta_k|Z_k=t_k)\rightarrow 1$, $Var(\beta_k|Z_k=t_k)\rightarrow 0$, as $k\rightarrow\infty$.
	Note that although the form of $C^{(k-1)}_2$ here depends upon the error function $\frac{x}{\log x}$, the actual form is not required for the proof.
\end{proof}
The following remark seems to be quite important from the number-theoretic perspective.
\begin{remark}
	\label{remark:remark1}
	Theorem \ref{theorem:conv2} 
	asserts that in the ``in probability" sense, $\Lambda_{\alpha,\beta}((0,x])$ must be $Li(x)+O\left(\frac{x}{\log x}\right)$, as $x\rightarrow\infty$.
        Now replacing $\alpha$ and $\beta$ by the recursive parameter notation $\alpha_k$ and $\beta_k$, observe that the randomness of $\alpha_k$ and $\beta_k$
	includes the true (deterministic) prime-counting function $\varphi(x)$ in the support of our random intensity measure. 
	Hence, the above argument 
	suggests that
	\begin{equation}
		\varphi(x)=Li(x)+O\left(\frac{x}{\log x}\right),~\mbox{as}~x\rightarrow\infty.
		\label{eq:pi1}
	\end{equation}
	Although the above deterministic result (\ref{eq:pi1}) has not been proved by us, it is a result that received strong evidence by our Bayesian asymptotic theory. 
\end{remark}
\begin{remark}
	\label{eqmark:remark_another}
Now what happens if we wish to consider the theory in Section \ref{sec:another} to similarly construct an error bound?	Here, at stage $k$, it is
given by $E(\beta_k|Z_k=t_k)\times t^{1/2+\epsilon}_k=\frac{k}{t^{1/2+\epsilon}}\times t^{{1/2}+\epsilon}=k$, that is, $\int_2^x f(u)du=x$, which is clearly
a much larger error bound compared to $x/\log x$ derived above, and hence, not useful. Indeed, with this error bound, it is easily seen that
$\beta_k\stackrel{P}{\longrightarrow}0$.
\end{remark}

\section{Sharper error bound with Bayesian theory?}
\label{sec:sharp}
It has been proved in $1899$ by de la Vall\'{e}e Poussin (see Theorem 23 of \ctn{Ing2000}), that
\begin{equation*}
	\varphi(x)=Li(x)+O\left(x\exp\left\{-a\sqrt{\log x}\right\}\right)~\mbox{as}~x\rightarrow\infty,
\end{equation*}
for some $a>0$.
\ctn{Ford02} showed that
\begin{equation*}
	\varphi(x)=Li(x)+O\left(x\exp\left\{-0.2098\left(\log x\right)^{3/5}\left(\log\log x\right)^{-1/5}\right\}\right)~\mbox{as}~x\rightarrow\infty,
\end{equation*}
However, the sharpest and explicit bound has been obtained by \ctn{Moss15}, as follows:
\begin{equation}
	|\varphi(x)-Li(x)|\leq 0.2593\frac{x}{(\log x)^{3/4}}\exp\left(-\sqrt{\frac{\log x}{6.315}}\right)~\mbox{for}~x\geq 229.
	\label{eq:moss1}
\end{equation}
As such, we shall consider the bound in (\ref{eq:moss1}) for further investigation for our purpose.
This bound is of course established and requires no validation. Rather, here our purpose is to provide an evaluation of our Bayesian theory
and illustrate the possible ramifications.

With respect to this, we formalize our Bayesian 
result with the theorem below. 

\begin{theorem}
	\label{theorem:conv3}
	Consider an NHPP with intensity function
\begin{equation}
	\Lambda_{\alpha,\beta}(A)=\alpha\int_A li(u)du+\beta\int_A f(u)du,
	\label{eq:intensity_measure3}
\end{equation}
	where the function $f(\cdot)$ is such that 
	\begin{equation}
	\int_2^xf(u)du=\frac{x}{(\log x)^{3/4}}\exp\left(-\sqrt{\frac{\log x}{6.315}}\right).
		\label{eq:bound_moss}
	\end{equation}
	Let $\alpha_k$ and $\beta_k$ denote the stage-wise parameters corresponding to $\alpha$ and $\beta$ in (\ref{eq:intensity_measure3}).
	Then, with respect to the respective posteriors $\pi(\alpha_k|Z_k=t_k)$
	and $\pi(\beta_k|Z_k=t_k)$; $k\geq 1$, the following holds:
	\begin{equation}
		\alpha_k\stackrel{P}{\longrightarrow}1~\mbox{and}~\beta_k-E(\beta_k|Z_k=t_k)\stackrel{P}{\longrightarrow 0},~\mbox{as}~k\rightarrow\infty.
		\label{eq:conv3}
	\end{equation}
	where
	\begin{equation}
		\log\left[E(\beta_k|Z_k=t_k)\right]\sim\sqrt{\frac{\log k}{6.315}}-\frac{1}{4}\log\log k\rightarrow\infty,~\mbox{as}~k\rightarrow\infty.
		\label{eq:conv4}
	\end{equation}
\end{theorem}
\begin{proof}
That $\alpha_k\stackrel{P}{\longrightarrow}1~$ as $k\rightarrow\infty$ follows in the same way as before. For the result on $\beta_k$, note that 
	the expression for $E(\beta_k|Z_k=t_k)$ remains the same as (\ref{eq:postmean_beta_k})
	but with $\sum_{i=0}^{k-1}B^{(i)}_2=\frac{t_k}{(\log t_k)^{3/4}}\exp\left(-\sqrt{\frac{\log t_k}{6.315}}\right)$. The form of $C^{(k-1)}_2$ is now different
	since it is based on a different expression for $f(\cdot)$, but the exact form is not necessary for this proof.

	Using $t_k\sim k\log k$, it follows that as $k\rightarrow\infty$,
	\begin{align}
		&\mu^{(1)}_k=\frac{\xi+k-1}{b+\sum_{i=0}^{k-1}B^{(i)}_2}\sim \left(\log k\right)^{-1/4}\exp\left(\sqrt{\frac{\log t_k}{6.315}}\right);\notag\\ 
		&\mu^{(2)}_k=\frac{\xi+k}{b+\sum_{i=0}^{k-1}B^{(i)}_2}\sim \left(\log k\right)^{-1/4}\exp\left(\sqrt{\frac{\log t_k}{6.315}}\right).\notag 
	\end{align}
	Hence, for any $\epsilon>0$, there exists $k_0\geq 1$ such that for $k\geq k_0$,
	\begin{equation}
		(1-\epsilon)\left(\log k\right)^{-1/4}\exp\left(\sqrt{\frac{\log t_k}{6.315}}\right)<\mu^{(i)}_k
		<(1+\epsilon)\left(\log k\right)^{-1/4}\exp\left(\sqrt{\frac{\log t_k}{6.315}}\right),~\mbox{for}~i=1,2.
		\label{eq:beta_new1}
	\end{equation}
	Applying (\ref{eq:beta_new1}) to the expression for $E(\beta_k|Z_k=t_k)$, it follows that
	\begin{equation}
		E(\beta_k|Z_k=t_k)\sim \left(\log k\right)^{-1/4}\exp\left(\sqrt{\frac{\log t_k}{6.315}}\right),~\mbox{as}~k\rightarrow\infty.
		\label{eq:beta_new2}
	\end{equation}
	Taking log of the right hand side of (\ref{eq:beta_new2}), which we denote by $\nu_k$, gives
	\begin{equation}
		\log\nu_k=\sqrt{\frac{\log t_k}{6.315}}-\frac{1}{4}\log\log k\sim \sqrt{\frac{\log k}{6.315}}-\frac{1}{4}\log\log k.
		\label{eq:beta_new3}
	\end{equation}
	The second asymptotic equivalence holds because $t_k\sim k\log k$ and 
	$$\frac{\log\log k}{\sqrt{\frac{\log k}{6.315}}}\rightarrow 0,~\mbox{as}~k\rightarrow\infty.$$
	It follows from (\ref{eq:beta_new3}) that
	$$\log \left[E(\beta_k|Z_k=t_k)\right]\sim \sqrt{\frac{\log k}{6.315}}-\frac{1}{4}\log\log k\rightarrow\infty,~\mbox{as}~k\rightarrow\infty.$$

The rest of the proof follows in the same way as that of Theorem \ref{theorem:pconv_beta}.
\end{proof}

	Theorem \ref{theorem:conv3} shows that $\beta_k$ are not even bounded in probability with respect to $\pi(\beta_k|Z_k=t_k)$, as $k\rightarrow\infty$.
	However, the rate of divergence of $E(\beta_k|Z_k=t_k)$ to infinity, given by 
$\left(\log k\right)^{-1/4}\exp\left(\sqrt{\frac{\log t_k}{6.315}}\right)$, is
	``extremely" slow, particularly with respect to $E(\beta_k|Z_k=t_k)\sim \frac{\sqrt{k}}{\left(\log k\right)^{3/2}}$, associated with the error bound 
	$\sqrt{x}\log x$, pertaining to RH. Indeed, Table \ref{table:compare} provides a comparison of the asymptotic forms for some large values of $k$,
explicitly pointing towards the enormity of the difference in the divergence speeds.
	
\begin{table}
	\caption{Comparison of the asymptotic forms of $E(\beta_k|Z_k=t_k)$.}
\label{table:compare}
\begin{center}
\begin{tabular}{|c||c|c|}\hline
	$k$ & $\frac{\sqrt{k}}{\left(\log k\right)^{3/2}}$ &  $\left(\log k\right)^{-1/4}\exp\left(\sqrt{\frac{\log t_k}{6.315}}\right)$\\
\hline
	$10^{10}$ & $905.058$ & $3.081$\\
	$10^{50}$ & $8.095\times 10^{21}$ & $21.830$\\
	$10^{100}$ & $2.862\times 10^{46}$ & $107.618$\\
	$10^{150}$ & $1.558\times 10^{71}$ & $377.783$\\
	$10^{200}$ & $1.012\times 10^{96}$ & $1103.769$\\
	$10^{250}$ & $7.240\times 10^{120}$ & $2860.237$\\
	$10^{300}$ & $5.508\times 10^{145}$ & $6797.735$\\
	$10^{350}$ & $4.371\times 10^{170}$ & $15120.430$\\
	$10^{400}$ & $3.578\times 10^{195}$ & $31901.430$\\
	$10^{450}$ & $2.998\times 10^{220}$ & $64443.220$\\
	$10^{500}$ & $2.560\times 10^{245}$ & $125503.700$\\
\hline
\end{tabular}
\end{center}
\end{table}

\begin{remark}
	\label{remark:remark3}
Multiplying the rate of divergence $(\log t_k)^{-1/4}\exp\left(\sqrt{\frac{\log t_k}{6.315}}\right)$ with the error bound 
	$\frac{t_k}{(\log t_k)^{3/4}}\exp\left(-\sqrt{\frac{\log t_k}{6.315}}\right)$ gives $\frac{t_k}{\log t_k}\sim \frac{k}{\log k}$.
	Thus, the new adjusted error bound become $\frac{x}{\log x}$, which the same as that considered in Theorem \ref{theorem:conv2}, for which 
	$\beta_k\stackrel{P}{\longrightarrow}1$, as $k\rightarrow\infty$, re-asserting the result (\ref{eq:pi1}). It thus seems that sharper bounds
	for $\varphi(x)$ are not possible with our Bayesian theory.
\end{remark}

\section{Bayesian asymptotic model comparison}
\label{sec:compare}
Although we derive $O\left(\frac{x}{\log x}\right)$ as a valid error bound, the smaller upper bound  
$\frac{x}{(\log x)^{3/4}}\exp\left(-\sqrt{\frac{\log x}{6.315}}\right)$ is not unacceptable, since $E(\beta_k|Z_k=t_k)$ diverges at an extremely slow rate.
The issue then raises the question as to which upper bound to exploit in the Poisson process model.
The smaller bound, if acceptable, is certainly preferable, and we indeed validate this with asymptotic theory of our recursive posterior predictive distributions.

Note that (\ref{eq:post_pred_k}) admits the following representation:
\begin{align}
	&\pi(Z_{k+1}=t_{k+1}|Z_k=t_k)\notag\\	
	&\qquad
	=\frac{\Gamma(\gamma+k-1)}{\left(a+\sum_{i=0}^{k}B^{(i)}_1\right)^{\gamma+k-1}} 
	\times\frac{\Gamma(\xi+k-1)}{\left(b+\sum_{i=0}^{k}B^{(i)}_2\right)^{\xi+k-1}}\notag\\
	&\qquad \times\left[R_kC^{(k)}_1C^{(k-1)}_1\times
	\frac{\gamma+k}{\left(a+\sum_{i=0}^{k}B^{(i)}_1\right)}\times  \frac{\gamma+k-1}{\left(a+\sum_{i=0}^{k}B^{(i)}_1\right)}\right.\notag\\ 
	&\qquad\left. +R_kC^{(k)}_2C^{(k-1)}_1\times
	\frac{\gamma+k-1}{\left(a+\sum_{i=0}^{k}B^{(i)}_1\right)} 
	\times\frac{\xi+k-1}{\left(b+\sum_{i=0}^{k}B^{(i)}_2\right)}\right.\notag\\
	&\qquad \left. +R_kC^{(k)}_1C^{(k-1)}_2\times
	\frac{\gamma+k-1}{\left(a+\sum_{i=0}^{k}B^{(i)}_1\right)} 
	\times\frac{\xi+k-1}{\left(b+\sum_{i=0}^{k}B^{(i)}_2\right)}\right.\notag\\
	&\qquad \left. +R_kC^{(k)}_2C^{(k-1)}_2\times
	\frac{\xi+k}{\left(b+\sum_{i=0}^{k}B^{(i)}_2\right)} 
	\times\frac{\xi+k-1}{\left(b+\sum_{i=0}^{k}B^{(i)}_2\right)}\right].
	\label{eq:post_pred2}
\end{align}
When the error term is $\frac{x}{(\log x)^{3/4}}\exp\left(-\sqrt{\frac{\log x}{6.315}}\right)$, 
\begin{equation}
\sum_{i=0}^kB^{(i)}_2=\frac{t_{k+1}}{(\log t_{k+1})^{3/4}}\exp\left(-\sqrt{\frac{\log t_{k+1}}{6.315}}\right),
	\label{eq:bound_M1}
\end{equation}
and when the error bound is $\frac{x}{\log x}$, then
\begin{equation}
	\sum_{i=0}^kB^{(i)}_2=\frac{t_{k+1}}{\log t_{k+1}}.
	\label{eq:bound_M2}
\end{equation}
We denote the models associated with (\ref{eq:bound_M1}) and (\ref{eq:bound_M2}) by $\mathcal M_1$ and $\mathcal M_2$, respectively.
Let the associated posterior predictive densities at $t=t_{k+1}$ be denoted by $\pi(Z_{k+1}=t_{k+1}|Z_k=t_k,\mathcal M_1)$ and
$\pi(Z_{k+1}=t_{k+1}|Z_k=t_k,\mathcal M_2)$, respectively. Then we have the following result 
\begin{theorem}
	\label{theorem:postpred}
$$\frac{\pi(Z_{k+1}=t_{k+1}|Z_k=t_k,\mathcal M_1)}{\pi(Z_{k+1}=t_{k+1}|Z_k=t_k,\mathcal M_2)}
	\rightarrow\infty,~\mbox{as}~k\rightarrow\infty.$$
\end{theorem}
\begin{proof}
Note that
\begin{equation}
	\sum_{i=0}^kB^{(i)}_1\sim k,~\mbox{as}~k\rightarrow\infty,
	\label{eq:B1}
\end{equation}
for both $\mathcal M_1$ and $\mathcal M_2$. In fact, $\sum_{i=0}^kB^{(i)}_1$ is the same for both $\mathcal M_1$ and $\mathcal M_2$, for any $k\geq 1$.
	Now, applying (\ref{eq:B1}) and (\ref{eq:bound_M1}) to (\ref{eq:post_pred2}) yields, as $k\rightarrow\infty$,
\begin{align}
	&\pi(Z_{k+1}=t_{k+1}|Z_k=t_k,\mathcal M_1)\notag\\
	&\qquad
	\sim\frac{\Gamma(\gamma+k-1)}{\left(a+\sum_{i=0}^{k}B^{(i)}_1\right)^{\gamma+k-1}} 
	\times\frac{\Gamma(\xi+k-1)}{\left(b+\sum_{i=0}^{k}B^{(i)}_2\right)^{\xi+k-1}}\notag\\
	&\qquad \times\left[R_kC^{(k)}_1C^{(k-1)}_1\times 1 
	 +R_kC^{(k)}_2C^{(k-1)}_1\times
	\left(\log k\right)^{3/4}\exp\left(\sqrt{\frac{\log t_{k+1}}{6.315}}\right)\right.\notag\\
	&\qquad \left. +R_kC^{(k)}_1C^{(k-1)}_2\times
	\left(\log k\right)^{3/4}\exp\left(\sqrt{\frac{\log t_{k+1}}{6.315}}\right)
	 +R_kC^{(k)}_2C^{(k-1)}_2\times
	\left(\log k\right)^{3/2}\exp\left(2\sqrt{\frac{\log t_{k+1}}{6.315}}\right)\right]\notag\\
	&\qquad\geq 
	\frac{\Gamma(\gamma+k-1)}{\left(a+\sum_{i=0}^{k}B^{(i)}_1\right)^{\gamma+k-1}} 
	\times\frac{\Gamma(\xi+k-1)}{\left(b+\sum_{i=0}^{k}B^{(i)}_2\right)^{\xi+k-1}}\notag\\
	&\qquad= 
	\frac{\Gamma(\gamma+k-1)}{\left(a+\sum_{i=0}^{k}B^{(i)}_1\right)^{\gamma+k-1}} 
	\times\frac{\Gamma(\xi+k-1)}{\left(b+ \left(\log k\right)^{3/4}\exp\left(\sqrt{\frac{\log t_{k+1}}{6.315}}\right)\right)^{\xi+k-1}}.
	\label{eq:post_pred3}
\end{align}
	On the other hand, applying (\ref{eq:B1}) and (\ref{eq:bound_M2}) to (\ref{eq:post_pred2}) yields
\begin{align}
	&\pi(Z_{k+1}=t_{k+1}|Z_k=t_k,\mathcal M_2)\notag\\
	&\qquad
	\sim\frac{\Gamma(\gamma+k-1)}{\left(a+\sum_{i=0}^{k}B^{(i)}_1\right)^{\gamma+k-1}} 
	\times\frac{\Gamma(\xi+k-1)}{\left(b+\sum_{i=0}^{k}B^{(i)}_2\right)^{\xi+k-1}}\notag\\
	&\qquad \times\left[R_kC^{(k)}_1C^{(k-1)}_1\times 1 
	 +R_kC^{(k)}_2C^{(k-1)}_1\times
	\log k\right.\notag\\
	&\qquad \left. +R_kC^{(k)}_1C^{(k-1)}_2\times
	\log k
	 +R_kC^{(k)}_2C^{(k-1)}_2\times
	\left(\log k\right)^2\right]\notag\\
	&\qquad\leq 
	\frac{\Gamma(\gamma+k-1)}{\left(a+\sum_{i=0}^{k}B^{(i)}_1\right)^{\gamma+k-1}} 
	\times\frac{\Gamma(\xi+k-1)}{\left(b+\sum_{i=0}^{k}B^{(i)}_2\right)^{\xi+k-1}}\times\left(\log k\right)^2\notag\\
	&\qquad= 
	\frac{\Gamma(\gamma+k-1)}{\left(a+\sum_{i=0}^{k}B^{(i)}_1\right)^{\gamma+k-1}} 
	\times\frac{\Gamma(\xi+k-1)}{\left(b+ \left(\frac{t_k}{\log t_k}\right)\right)^{\xi+k-1}}.
	\label{eq:post_pred4}
\end{align}
Combining (\ref{eq:post_pred3}) and (\ref{eq:post_pred4}) gives, for sufficiently large $k$,
\begin{align}
\frac{\pi(Z_{k+1}=t_{k+1}|Z_k=t_k,\mathcal M_1)}{\pi(Z_{k+1}=t_{k+1}|Z_k=t_k,\mathcal M_2)}
	&\geq\left(\frac{\left(b+ \left(\frac{t_k}{\log t_k}\right)\right)}{\left(b+ \left(\log k\right)^{3/4}\exp\left(\sqrt{\frac{\log t_{k+1}}{6.315}}\right)\right)}\right)^{\xi+k-1}\notag\\
	&\rightarrow\infty,~\mbox{as}~k\rightarrow\infty.\notag
\end{align}
\end{proof}

\section{Asymptotic equivalence of recursive and non-recursive Bayesian inferences}
\label{sec:rec_nonrec}
In this section, we derive the posterior distributions of the parameters and the posterior predictive distribution using traditional, non-recursive Bayesian theory and
establish that they asymptotically lead to the same results as their recursive counterparts. However, from the practical computation standpoint, the non-recursive approach
is infeasible, since the posteriors and posterior predictive distributions consist of $2^k$ and $2^{k+1}$ terms, respectively, given the first $k$ prime numbers. 
That is, conditional on a million prime numbers, these distributions consist of the order of $2^{\mbox{million}}$ terms, which can not be dealt with in practice. 
In stark contrast, the recursive posteriors of the parameters and the recursive posterior predictive distributions consist of only two and four terms, respectively, 
for every $k\geq 1$.

\subsection{Non-recursive posterior}
\label{subsec:nonrecursive_posterior}
Assuming $\alpha$ and $\beta$ do not directly depend upon the stage, for $t_1,\ldots,t_k$, the likelihood is given by
\begin{align}
&L_k(\alpha,\beta)=\prod_{i=0}^{k-1}\left[\exp\left\{-\alpha B^{(i)}_1-\beta B^{(i)}_2\right\}\left(\alpha C^{(i)}_1+\beta C^{(i)}_2\right)\right]\notag\\
&=\exp\left\{-\alpha\sum_{i=1}^{k-1}B^{(i)}_1-\beta\sum_{i=0}^{k-1}B^{(i)}_2\right\}\notag\\
&\qquad\times\sum_{r=0}^{k}\alpha^r\beta^{k-r}\left(\sum_{i_1,i_2,\ldots,i_k}C^{(i_1)}_1C^{(i_2)}_1\cdots C^{i_r}_1C^{(i_{r+1})}_2C^{(i_{r+2})}_2\cdots C^{(i_{k})}_2\right).
	\label{eq:lik1}
\end{align}
In (\ref{eq:lik1}), $i_1,i_2,\ldots,i_k\in\{0,1,\ldots,k-1\}$ and the sum $\Sigma_{i_1,i_2,\ldots,i_k}$ is over $k\choose r$ terms.
With $$\pi(\alpha,\beta)\propto \exp(-a\alpha)\alpha^{\gamma-1}\times \exp(-b\beta)\beta^{\xi-1},$$ the posterior of $\alpha,\beta$ is given by
\begin{align}
	&	\pi(\alpha,\beta|t_1,\ldots,t_k)\notag\\
	&\propto\exp\left\{-\alpha\sum_{i=1}^{k-1}B^{(i)}_1-\beta\sum_{i=0}^{k-1}B^{(i)}_2i-a-b\right\}\notag\\
	&\qquad\times\sum_{r=0}^{k}\left[\alpha^{r+\gamma-1}\beta^{k-r+\xi-1}\left(\sum_{i_1,i_2,\ldots,i_k}C^{(i_1)}_1C^{(i_2)}_1\cdots C^{i_r}_1C^{(i_{r+1})}_2C^{(i_{r+2})}_2\cdots C^{(i_{k})}_2\right)\right]\notag\\
	&=
	\sum_{r=0}^{k}\left[g(\alpha:a+\sum_{i=0}^{k-1}B^{(i)}_1,\gamma+r)\times g(\beta:b+\sum_{i=0}^{k-1}B^{(i)}_2,\xi+k-r)\right.\notag\\
	&\qquad\left.\times\left(\sum_{i_1,i_2,\ldots,i_k}C^{(i_1)}_1C^{(i_2)}_1\cdots C^{i_r}_1C^{(i_{r+1})}_2C^{(i_{r+2})}_2\cdots C^{(i_{k})}_2\right)\right.\notag\\
	&\qquad\left.\times\frac{\Gamma(\gamma+r)}{\left(a+\sum_{i=0}^{k-1}B^{(i)}_1\right)^{\gamma+r}}
	\times\frac{\Gamma(\xi+k-r)}{\left(b+\sum_{i=0}^{k-1}B^{(i)}_2\right)^{\xi+k-r}}\right].
	\label{eq:post1}
\end{align}
With the normalizing constant in (\ref{eq:post1}) given by
\begin{align}
	\tilde R^{-1}_k
	&=\sum_{r=0}^k\left[\left(\sum_{i_1,i_2,\ldots,i_k}C^{(i_1)}_1C^{(i_2)}_1\cdots C^{i_r}_1C^{(i_{r+1})}_2C^{(i_{r+2})}_2\cdots C^{(i_{k})}_2\right)\right.\notag\\
	&\qquad\left.\times\frac{\Gamma(\gamma+r)}{\left(a+\sum_{i=0}^{k-1}B^{(i)}_1\right)^{\gamma+r}}
	\times\frac{\Gamma(\xi+k-r)}{\left(b+\sum_{i=0}^{k-1}B^{(i)}_2\right)^{\xi+k-r}}\right],
	\label{eq:normconst1}
\end{align}
the posterior has the exact form 
\begin{equation}
	\pi(\alpha,\beta|t_1,\ldots,t_k)=\sum_{r=0}^k p_{rk}g(\alpha:a+\sum_{i=0}^{k-1}B^{(i)}_1,\gamma+r)\times g(\beta:b+\sum_{i=0}^{k-1}B^{(i)}_2,\xi+k-r) 
	\label{eq:post2}
\end{equation}
where 
\begin{align}
	p_{rk} &=\tilde R_k
	\times\left(\sum_{i_1,i_2,\ldots,i_k}C^{(i_1)}_1C^{(i_2)}_1\cdots C^{i_r}_1C^{(i_{r+1})}_2C^{(i_{r+2})}_2\cdots C^{(i_{k})}_2\right)\notag\\
	&\qquad\qquad\times\frac{\Gamma(\gamma+r)}{\left(a+\sum_{i=0}^{k-1}B^{(i)}_1\right)^{\gamma+r}}
	\times\frac{\Gamma(\xi+k-r)}{\left(b+\sum_{i=0}^{k-1}B^{(i)}_2\right)^{\xi+k-r}}.
	\label{eq:post3}
\end{align}
It is evident from (\ref{eq:post3}) that $p_{rk}>0$ for $r=0,1,\ldots,k$ for all $k\geq 1$ and $\sum_{r=0}^kp_{rk}=1$ for all $k\geq 1$.
Thus, (\ref{eq:post2}) is a mixture of $g(\alpha:a+\sum_{i=0}^{k-1}B^{(i)}_1,\gamma+r)\times g(\beta:b+\sum_{i=0}^{k-1}B^{(i)}_2,\xi+k-r)$ with $\{p_{rk}:r=0,1,\ldots,k\}$
being the mixing probabilities.

\begin{remark}
	\label{remark:non_recursive}
	Note that (\ref{eq:post1}) consists of $2^k$ terms, unlike only $2$ terms in the recursive counterpart (\ref{eq:post_k}).
\end{remark}

\subsection{Convergence of $\alpha$ with respect to non-recursive posterior}
\label{subsec:nonrec_alpha}
\begin{theorem}
	\label{theorem:probconv_alpha}
With respect to the sequence of marginal probability measures $\pi(\alpha|t_1,\ldots,t_k)$; $k\geq 1$,
\begin{equation}
	\alpha\stackrel{P}{\longrightarrow}1,~\mbox{as}~k\rightarrow\infty.
	\label{eq:probconv_alpha}
\end{equation}
\end{theorem}
\begin{proof}
It follows from (\ref{eq:post2}) that 
\begin{align}
	E(\alpha|t_1,\ldots,t_k)&=\sum_{r=0}^kp_{rk}\frac{\gamma+r}{a+\sum_{i=0}^{k-1}B^{(i)}_1}; \label{eq:mean_alpha}\\
	E(\alpha^2|t_1,\ldots,t_k)&=\sum_{r=0}^kp_{rk}\frac{(\gamma+r)(\gamma+r+1)}{\left(a+\sum_{i=0}^{k-1}B^{(i)}_1\right)^2}; \label{eq:meansq_alpha}
\end{align}
Note that for any given $r_0\geq 0$, as $k\rightarrow\infty$,
\begin{align}
	&\sum_{r=0}^{r_0}p_{rk}\rightarrow 0;\label{eq:p1}\\
	&\sum_{r=0}^{r_0}p_{rk}\frac{\gamma+r}{a+\sum_{i=0}^{k-1}B^{(i)}_1}\rightarrow 0;\label{eq:p2}\\
	&\sum_{r=0}^{r_0}p_{rk}\frac{(\gamma+r)(\gamma+r+1)}{\left(a+\sum_{i=0}^{k-1}B^{(i)}_1\right)^2}\rightarrow 0.\label{eq:psq}
\end{align}
Due to (\ref{eq:p2}), it is sufficient to show that for some $r_0$ sufficiently large,
\begin{equation}
	\sum_{r=r_0}^kp_{rk}\frac{\gamma+r}{a+\sum_{i=0}^{k-1}B^{(i)}_1}\rightarrow 1,~\mbox{as}~k\rightarrow\infty.
	\label{eq:p3}
\end{equation}
To this end, let us first consider a constant $0<K<k$ and in $\sum_{r=r_0}^kp_{rk}\frac{\gamma+r}{a+\sum_{i=0}^{k-1}B^{(i)}_1}$, let us replace
$\gamma+r$ with $\gamma+K-\frac{1}{r}$. So, as $r\rightarrow\infty$,  $\gamma+K-\frac{1}{r}\uparrow \gamma+K$, which imitates the divergence behaviour of $\gamma+r$
as $r\rightarrow\infty$, when $K\rightarrow\infty$ as a function of $r$.  

Now, since $\gamma+K-\frac{1}{r}\rightarrow K$, for any $\epsilon>0$, there exists $r_0\geq 1$ such that for $r\geq r_0$, $K-\epsilon<K-\frac{1}{r}<K+\epsilon$
This implies 
\begin{equation}
	(\gamma+K-\epsilon)\left(1-\sum_{r=0}^{r_0}p_{rk}\right)<\sum_{r=r_0}^kp_{rk}\left(\gamma+K-\frac{1}{r}\right)<\sum_{r=r_0}^kp_{rk}(\gamma+K+\epsilon)<\gamma+K+\epsilon.
	\label{eq:bounds1}
\end{equation}
Due to (\ref{eq:p1}), for any $\eta>0$, there exists $k_0\geq 1$ such that $-\eta<\sum_{r=0}^{r_0}p_{rk}<\eta$, for $k\geq k_0$. 
Applying this to (\ref{eq:bounds1}) we obtain, for $k\geq k_0$,
\begin{equation}
	\gamma+K-\epsilon<(\gamma+K-\epsilon)\left(1+\eta\right)<\sum_{r=r_0}^kp_{rk}\left(\gamma+K-\frac{1}{r}\right)<\sum_{r=r_0}^kp_{rk}(\gamma+K+\epsilon)<\gamma+K+\epsilon.
	\label{eq:bounds2}
\end{equation}
Now note that replacing $\gamma+r$ with $\gamma+K-\frac{1}{r}$ in (\ref{eq:p2}) also gives
\begin{equation}
	\sum_{r=0}^{r_0}p_{rk}\left(\gamma+K-\frac{1}{r}\right)\rightarrow 0,~\mbox{as}~k\rightarrow\infty,
	\label{eq:bounds3}
\end{equation}
which holds because, for $r=0,1,\ldots,r_0$, $p_{rk}\rightarrow 0$ as $k\rightarrow\infty$.

It follows from (\ref{eq:bounds2}) and (\ref{eq:bounds3}) that for given $K>0$,
\begin{equation}
	\sum_{r=0}^kp_{rk}\left(\gamma+K-\frac{1}{r}\right)\rightarrow \gamma+K,~\mbox{as}~k\rightarrow\infty.
	\label{eq:bounds4}
\end{equation}
Now, replacing $k$ with $K-\frac{1}{k}$ in $\sum_{i=0}^{k-1}B^{(i)}_1$ shows that
\begin{equation}
	\sum_{i=0}^{K-\frac{1}{k}-1}B^{(i)}_1\rightarrow \sum_{i=0}^{K-1}B^{(i)}_1,~\mbox{as}~k\rightarrow\infty.
	\label{eq:bounds5}
\end{equation}
From (\ref{eq:bounds4}) and (\ref{eq:bounds5}) we see that
\begin{equation}
	 E\left(\alpha|t_1,\ldots,t_{K-\frac{1}{k}}\right)
	=\sum_{r=0}^kp_{rk}\frac{\gamma+K-\frac{1}{r}}{\sum_{i=0}^{K-\frac{1}{k}-1}B^{(i)}_1}\rightarrow\frac{\gamma+K}{\sum_{i=0}^{K-1}B^{(i)}_1},~\mbox{as}~k\rightarrow\infty.
	\label{eq:bounds6}
\end{equation}
Now, as $K\rightarrow\infty$, $\sum_{i=0}^{K-1}B^{(i)}_1\sim K$, so that $\frac{\gamma+K}{\sum_{i=0}^{K-1}B^{(i)}_1}\sim 1$, showing that
\begin{equation}
	\sum_{r=0}^kp_{rk}\frac{\gamma+r}{\sum_{i=0}^{k-1}B^{(i)}_1}\rightarrow 1,~\mbox{as}~k\rightarrow\infty.
	\label{eq:bounds7}
\end{equation}
In other words, 
\begin{equation}
	E(\alpha|t_1,\ldots,t_k)\rightarrow 1,~\mbox{as}~k\rightarrow\infty.
	\label{eq:mean_conv1}
\end{equation}
Using the same technique for (\ref{eq:meansq_alpha}) and (\ref{eq:psq}), it is seen that 
\begin{align}
E(\alpha^2|t_1,\ldots,t_{K-\frac{1}{k}})
	&=\sum_{r=0}^kp_{rk}\frac{\left(\gamma+K-\frac{1}{r}\right)\left(\gamma+K-\frac{1}{r}+1\right)}
{\left(a+\sum_{i=0}^{Ki-\frac{1}{k}-1}B^{(i)}_1\right)^2}\notag\\
	&\rightarrow\frac{(\gamma+K)(\gamma+K+1)}{\left(a+\sum_{i=0}^{K-1}B^{(i)}_1\right)^2},~\mbox{as}~k\rightarrow\infty.
\label{eq:var_alpha1}
\end{align}
From (\ref{eq:var_alpha1}) and (\ref{eq:bounds6}) it follows that as $k\rightarrow\infty$,
\begin{align}
	&Var\left(\alpha|t_1,\ldots,t_{K-\frac{1}{k}}\right)=E(\alpha^2|t_1,\ldots,t_{K-\frac{1}{k}})- \left[E(\alpha|t_1,\ldots,t_{K-\frac{1}{k}})\right]^2\notag\\
	&\qquad\rightarrow \frac{(\gamma+K)(\gamma+K+1)}{\left(a+\sum_{i=0}^{K-1}B^{(i)}_1\right)^2}-\left[\frac{\gamma+K}{a+\sum_{i=0}^{K-1}B^{(i)}_1}\right]^2
	=\frac{\gamma+K}{\left(a+\sum_{i=0}^{K-1}B^{(i)}_1\right)^2}.
	\label{eq:var_alpha2}
\end{align}
Taking $K\rightarrow\infty$ in (\ref{eq:var_alpha2}) gives
\begin{equation}
Var(\alpha|t_1,\ldots,t_k)\rightarrow\lim_{k\rightarrow\infty}\frac{\gamma+k}{a+\left(\sum_{i=0}^{k-1}B^{(i)}_1\right)^2}=0,
	\label{eq:var_alpha3}
\end{equation}
since $\sum_{i=0}^{k-1}B^{(i)}_1\sim k$, as $k\rightarrow\infty$.

From (\ref{eq:mean_conv1}) and (\ref{eq:var_alpha1}) we assert that with respect to the sequence of probability measures $\pi(\alpha|t_1,\ldots,t_k)$; $k\geq 1$,
(\ref{eq:probconv_alpha}) holds.
\end{proof}

\subsection{Convergence of $\beta$ with respect to non-recursive posterior}
\label{subsec:nonrec_beta}

Now, expanding the product $\prod_{i=0}^{k-1}\left(\alpha C^{(i)}_1+\beta C^{(i)}_2\right)$ in terms of $\beta^r\alpha^{k-r}$, for $r=0,1,\ldots,k$ gives 
\begin{equation}
	\pi(\alpha,\beta|t_1,\ldots,t_k)=\sum_{r=0}^k \tilde q_rg(\alpha:a+\sum_{i=0}^{k-1}B^{(i)}_1,\xi+k-r)\times g(\beta:b+\sum_{i=0}^{k-1}B^{(i)}_2,\gamma+r), 
	\label{eq:post4}
\end{equation}
with
\begin{align}
	q_r &=Q_k
	\times\left(\sum_{i_1,i_2,\ldots,i_k}C^{(i_1)}_2C^{(i_2)}_2\cdots C^{i_r}_2C^{(i_{r+1})}_1C^{(i_{r+2})}_1\cdots C^{(i_{k})}_1\right)\notag\\
	&\qquad\qquad\times\frac{\Gamma(\gamma+r)}{\left(a+\sum_{i=0}^{k-1}B^{(i)}_1\right)^{\xi+k-r}}
	\times\frac{\Gamma(\xi+k-r)}{\left(b+\sum_{i=0}^{k-1}B^{(i)}_2\right)^{\gamma+r}}.
	\label{eq:post5}
\end{align}
where
\begin{align}
	Q^{-1}_k
	&=\sum_{r=0}^k\left[\left(\sum_{i_1,i_2,\ldots,i_k}C^{(i_1)}_2C^{(i_2)}_2\cdots C^{i_r}_2C^{(i_{r+1})}_1C^{(i_{r+2})}_1\cdots C^{(i_{k})}_1\right)\right.\notag\\
	&\qquad\left.\times\frac{\Gamma(\gamma+r)}{\left(a+\sum_{i=0}^{k-1}B^{(i)}_1\right)^{\xi+k-r}}
	\times\frac{\Gamma(\xi+k-r)}{\left(b+\sum_{i=0}^{k-1}B^{(i)}_2\right)^{\gamma+r}}\right].
	\label{eq:normconst2}
\end{align}

\begin{theorem}
	\label{theorem:probconv_beta}
With respect to the sequence of marginal probability measures $\pi(\beta|t_1,\ldots,t_k)$; $k\geq 1$,
\begin{equation}
	\beta-E(\beta|t_1,\ldots,t_k)\stackrel{P}{\longrightarrow}0,~\mbox{as}~k\rightarrow\infty,
	\label{eq:probconv_beta}
\end{equation}
	where $E(\beta|t_1,\ldots,t_k)$ has the same asymptotic form as the corresponding recursive Bayesian theory with error function $f(\cdot)$.
\end{theorem}
\begin{proof}
It follows from (\ref{eq:post4}) that
\begin{align}
	&E(\beta|t_1,\ldots,t_k)=\sum_{r=0}^kq_r\frac{\xi+r}{b+\sum_{i=0}^{k-1}B^{(i)}_2}; \label{eq:mean_beta}\\
	&E(\beta^2|t_1,\ldots,t_k)=\sum_{r=0}^kq_r\frac{(\xi+r)(\xi+r+1)}{\left(b+\sum_{i=0}^{k-1}B^{(i)}_2\right)^2}. \label{eq:meansq_beta}
\end{align}
	Here we also have, as $k\rightarrow\infty$, fo any $r_0\geq 0$,
\begin{align}
	&\sum_{r=0}^{r_0}q_{rk}\rightarrow 0;\label{eq:q1}\\
	&\sum_{r=0}^{r_0}q_{rk}\frac{\xi+r}{b+\sum_{i=0}^{k-1}B^{(i)}_2}\rightarrow 0;\label{eq:q2}\\
	&\sum_{r=0}^{r_0}q_{rk}\frac{(\xi+r)(\xi+r+1)}{\left(b+\sum_{i=0}^{k-1}B^{(i)}_2\right)^2}\rightarrow 0.\label{eq:qsq}
\end{align}
	Applying the same techniques of the proof of Theorem \ref{theorem:probconv_beta} to (\ref{eq:mean_beta}) and (\ref{eq:meansq_beta}) we obtain,
	as $k\rightarrow\infty$,
	\begin{align}
		&E\left(\beta|t_1,\ldots,t_{K-\frac{1}{k}}\right)
	=\sum_{r=0}^kq_{rk}\frac{\xi+K-\frac{1}{r}}{\sum_{i=0}^{K-\frac{1}{k}-1}B^{(i)}_2}\rightarrow\frac{\xi+K}{b+\sum_{i=0}^{K-1}B^{(i)}_2};\label{eq:beta_bounds1}\\
	&Var\left(\beta|t_1,\ldots,t_{K-\frac{1}{k}}\right)
	\rightarrow \frac{(\xi+K)(\xi+K+1)}{\left(b+\sum_{i=0}^{K-1}B^{(i)}_2\right)^2}-\left[\frac{\xi+K}{b+\sum_{i=0}^{K-1}B^{(i)}_2}\right]^2
	=\frac{\xi+K}{\left(b+\sum_{i=0}^{K-1}B^{(i)}_2\right)^2}.
	\label{eq:var_beta}
	\end{align}
	Now observe that as $K\rightarrow\infty$, for all the forms of the error function $f(\cdot)$ considered in this article,
	the right hand side of (\ref{eq:beta_bounds1}) has the same asymptotic form as the corresponding recursive Bayesian theory, while
	the right hand side of (\ref{eq:var_beta}) converges to zero. It follows that (\ref{eq:probconv_beta}) holds.
\end{proof}

\subsection{Asymptotic non-recursive model comparison}
\label{subsec:compare_nonrecursive}
Here the posterior predictive density at the point $t=t_{k+1}$ is given by
\begin{align}
&\pi(Z_{k+1}=t_{k+1}|Z_1=t_1,\ldots,Z_k=t_k)=\int_0^{\infty}\int_0^{\infty}f_{Z_{k+1},\alpha,\beta}(t_{k+1})\pi(\alpha,\beta|t_1,\ldots,t_k)\notag\\
&\qquad =C^{(k)}_1\sum_{r=0}^kp_{rk}\left(\frac{\gamma+r}{a+\sum_{i=0}^kB^{(i)}_1}\right)\left(\frac{a+\sum_{i=0}^{k-1}B^{(i)}_1}{a+\sum_{i=0}^kB^{(i)}_1}\right)^{\gamma+r}
	\left(\frac{b+\sum_{i=0}^{k-1}B^{(i)}_2}{b+\sum_{i=0}^kB^{(i)}_2}\right)^{\xi+k-r}.
	\label{eq:postpred0}\\
	&\qquad +C^{(k)}_2\sum_{r=0}^kq_{rk}\left(\frac{\xi+r}{b+\sum_{i=0}^kB^{(i)}_2}\right)\left(\frac{b+\sum_{i=0}^{k-1}B^{(i)}_2}{b+\sum_{i=0}^kB^{(i)}_2}\right)^{\xi+r}
	\left(\frac{a+\sum_{i=0}^{k-1}B^{(i)}_1}{a+\sum_{i=0}^kB^{(i)}_1}\right)^{\xi+k-r}.
	\label{eq:postpred1}
\end{align}
\begin{remark}
	Again note that (\ref{eq:postpred1}) consists of $2^{k+1}$ terms, in sharp contrast with (\ref{eq:post_pred_k}), the corresponding recursive counterpart,
	having $4$ terms only. 
\end{remark}

\begin{theorem}
	\label{theorem:postpred_lim}
\begin{equation}
	\pi(Z_{k+1}=t_{k+1}|Z_1=t_1,\ldots,Z_k=t_k)\sim c_1 C^{(k)}_1+c_2 C^{(k)}_2\left(\frac{k}{b+\sum_{i=0}^kB^{(i)}_2}\right),~\mbox{as}~k\rightarrow\infty,
	\label{eq:postpred_lim}
\end{equation}
where $0<c_1,c_2<1$.
\end{theorem}
\begin{proof}
Let us first study the limit of (\ref{eq:postpred0}).
Note that for given $r\geq 0$, as $k\rightarrow\infty$,
\begin{align}
	&\qquad p_{rk}\rightarrow 0;~\frac{\gamma+r}{a+\sum_{i=0}^kB^{(i)}_1}\rightarrow 0;\notag\\
	&\left(\frac{a+\sum_{i=0}^{k-1}B^{(i)}_1}{a+\sum_{i=0}^kB^{(i)}_1}\right)^{\gamma+r}~\mbox{and}~
\left(\frac{b+\sum_{i=0}^{k-1}B^{(i)}_2}{b+\sum_{i=0}^kB^{(i)}_2}\right)^{\xi+k-r}<1,\notag\\
\end{align}
using which we see as before that for any $r_0\geq 0$,
\begin{align}
	&\sum_{r=_0}^{r_0}p_{rk}\left(\frac{\gamma+r}{a+\sum_{i=0}^kB^{(i)}_1}\right)\left(\frac{a+\sum_{i=0}^{k-1}B^{(i)}_1}{a+\sum_{i=0}^kB^{(i)}_1}\right)^{\gamma+r}
	\left(\frac{b+\sum_{i=0}^{k-1}B^{(i)}_2}{b+\sum_{i=0}^kB^{(i)}_2}\right)^{\xi+k-r}\rightarrow 0,~\mbox{as}~k\rightarrow\infty.
	\label{eq:postpred2}
\end{align}
Hence, again we consider the limit, as $k\rightarrow\infty$, of 
\begin{align}
	&\sum_{r=r_0}^{k}p_{rk}\left(\frac{\gamma+r}{a+\sum_{i=0}^kB^{(i)}_1}\right)\left(\frac{a+\sum_{i=0}^{k-1}B^{(i)}_1}{a+\sum_{i=0}^kB^{(i)}_1}\right)^{\gamma+r}
	\left(\frac{b+\sum_{i=0}^{k-1}B^{(i)}_2}{b+\sum_{i=0}^kB^{(i)}_2}\right)^{\xi+k-r}.\notag
\end{align}
Considering the same technique as before, let us replace $r$ and $k$ in 
$$\left(\frac{\gamma+r}{a+\sum_{i=0}^kB^{(i)}_1}\right)\left(\frac{a+\sum_{i=0}^{k-1}B^{(i)}_1}{a+\sum_{i=0}^kB^{(i)}_1}\right)^{\gamma+r}
	\left(\frac{b+\sum_{i=0}^{k-1}B^{(i)}_2}{b+\sum_{i=0}^kB^{(i)}_2}\right)^{\xi+k-r}$$ 
	by $K-\frac{1}{r}$ and $K-\frac{1}{k}$, respectively, where $0<K<\infty$, and take the limit
	as $r\rightarrow\infty$ of the resultant expression given by
\begin{equation}
	\phi(r,k)=\left(\frac{\gamma+K-\frac{1}{r}}{a+\sum_{i=0}^{K-\frac{1}{k}}B^{(i)}_1}\right)
	\left(\frac{a+\sum_{i=0}^{K-\frac{1}{k}-1}B^{(i)}_1}{a+\sum_{i=0}^{K-\frac{1}{k}}B^{(i)}_1}\right)^{\gamma+K-\frac{1}{r}}
	\left(\frac{b+\sum_{i=0}^{K-\frac{1}{k}-1}B^{(i)}_2}{b+\sum_{i=0}^{K-\frac{1}{k}}B^{(i)}_2}\right)^{\xi+K-\frac{1}{k}-K+\frac{1}{r}}.\notag
\end{equation}
Note that, since $r\leq k$, for $r\rightarrow\infty$, we must have $k\rightarrow\infty$. This yields
\begin{equation}
	\phi(r,k)\rightarrow 
	\left(\frac{\gamma+K}{a+\sum_{i=0}^{K}B^{(i)}_1}\right)
	\left(\frac{a+\sum_{i=0}^{K-1}B^{(i)}_1}{a+\sum_{i=0}^{K}B^{(i)}_1}\right)^{\gamma+K}
	\left(\frac{b+\sum_{i=0}^{K-1}B^{(i)}_2}{b+\sum_{i=0}^{K}B^{(i)}_2}\right)^{\xi},~\mbox{as}~r\rightarrow\infty.
	\label{eq:postpred5}
\end{equation}
It follows from (\ref{eq:postpred2}) and (\ref{eq:postpred5}), considering $K\rightarrow\infty$, that as $k\rightarrow\infty$,
\begin{align}
	&C^{(k)}_1\sum_{r=0}^{k}p_{rk}\left(\frac{\gamma+r}{a+\sum_{i=0}^kB^{(i)}_1}\right)\left(\frac{a+\sum_{i=0}^{k-1}B^{(i)}_1}{a+\sum_{i=0}^kB^{(i)}_1}\right)^{\gamma+r}
	\left(\frac{b+\sum_{i=0}^{k-1}B^{(i)}_2}{b+\sum_{i=0}^kB^{(i)}_2}\right)^{\xi+k-r}\notag\\
	&\sim
	 \lim_{k\rightarrow\infty}\left(\frac{a+\sum_{i=0}^{k-1}B^{(i)}_1}{a+\sum_{i=0}^{k}B^{(i)}_1}\right)^{k}C^{(k)}_1.
	\label{eq:postpred6}
\end{align}
In the same way, as $k\rightarrow\infty$, the limiting form of (\ref{eq:postpred1}) is given by 
\begin{align}
	&C^{(k)}_2\sum_{r=0}^{k}q_{rk}\left(\frac{\xi+r}{b+\sum_{i=0}^kB^{(i)}_2}\right)\left(\frac{b+\sum_{i=0}^{k-1}B^{(i)}_2}{b+\sum_{i=0}^kB^{(i)}_2}\right)^{\xi+r}
	\left(\frac{a+\sum_{i=0}^{k-1}B^{(i)}_1}{a+\sum_{i=0}^kB^{(i)}_1}\right)^{\gamma+k-r}\notag\\
	&\sim
	 \lim_{k\rightarrow\infty}\left(\frac{b+\sum_{i=0}^{k-1}B^{(i)}_2}{b+\sum_{i=0}^{k}B^{(i)}_2}\right)^{k}
	 \times C^{(k)}_2\left(\frac{k}{b+\sum_{i=0}^kB^{(i)}_2}\right).
	\label{eq:postpred7}
\end{align}
	Let us denote the limits $ \underset{k\rightarrow\infty}{\lim}~\left(\frac{a+\sum_{i=0}^{k-1}B^{(i)}_1}{a+\sum_{i=0}^{k}B^{(i)}_1}\right)^{k}$ and
	$ \underset{k\rightarrow\infty}{\lim}~\left(\frac{b+\sum_{i=0}^{k-1}B^{(i)}_2}{b+\sum_{i=0}^{k}B^{(i)}_2}\right)^{k}$ by $c_1$ and $c_2$, respectively.
Note that $0<c_1,c_2<1$. Hence, (\ref{eq:postpred_lim}) follows by combining (\ref{eq:postpred0}), (\ref{eq:postpred1}), (\ref{eq:postpred6}) and (\ref{eq:postpred7}). 
\end{proof}

Now, with respect to our non-recursive Bayesian theory, we wish to compare the models $\mathcal M_1$ and $\mathcal M_2$, which correspond 
to (\ref{eq:bound_M1}) and (\ref{eq:bound_M2}), respectively. We formalize our result in the following theorem.
\begin{theorem}
	\label{theorem:postpred_nonrecursive}
$$\frac{\pi(Z_{k+1}=t_{k+1}|Z_1,\ldots,Z_k=t_k,\mathcal M_1)}{\pi(Z_{k+1}=t_{k+1}|Z_1,\ldots,Z_k=t_k,\mathcal M_2)}
	\rightarrow\infty,~\mbox{as}~k\rightarrow\infty.$$
\end{theorem}
\begin{proof}
	The result follows by application of Theorem \ref{theorem:postpred_lim} to models $\mathcal M_1$ and $\mathcal M_2$.
\end{proof}

\section{Discovering candidate Mersenne primes using TMCMC}
\label{sec:disco_primes}

We shall begin with detailing our procedure for generating primes greater than any given integer. Then we shall extend the procedure to generate candidate
Mersenne primes or identify the candidate Mersenne prime exponents among those already generated.

\subsection{Bayesian prime generation procedure using TMCMC}
\label{subsec:ordinary_prime_generation}

Given primes $t_1,\ldots,t_k$ and any $t>t_k$, let $$B^{(k)}_1(t)=\int_{t_k}^tLi(u)du~\mbox{and}~B^{(k)}_2=\int_{t_k}^tf(u)du.$$
For large enough $k$, it follows from (\ref{eq:post_pred_k}) that we can simply work with the recursive posterior predictive density proportional to the following form:
\begin{align}
	h(t)&=\left(a+\sum_{i=0}^kB^{(i)}_1+B^{(k)}_1(t)\right)^{-(\gamma+k)}\times\left(b+\sum_{i=0}^kB^{(i)}_2+B^{(k)}_2(t)\right)^{-(\xi+k)}\notag\\
	&\qquad\times\left(li(t)li(t_k)+li(t_k)f(t)+li(t)f(t_k)+f(t)f(t_k)\right)\notag\\
	&=
	M_k\left(li(t)+f(t)\right)
	\left(a+\sum_{i=0}^kB^{(i)}_1+B^{(k)}_1(t)\right)^{-(\gamma+k)}\times\left(b+\sum_{i=0}^kB^{(i)}_2+B^{(k)}_2(t)\right)^{-(\xi+k)},
	\label{eq:tmcmc1}
\end{align}
with $M_k=li(t_k)+f(t_k)$.
Our goal is to simulate from $h(\cdot)$ using TMCMC and check whether or not the simulated TMCMC realization is a prime. 
Since our Bayesian asymptotic theory is valid for all $a\geq 0$, $b\geq 0$, $\gamma\geq 0$ and $\xi\geq 0$, we set $a=b=\gamma=xi=0$.
Also since $$\sum_{i=0}^kB^{(i)}_1+B^{(k)}_1(t)=Li(t)\sim\frac{t}{\log t}~\mbox{and}~\sum_{i=0}^kB^{(i)}_1+B^{(k)}_2(t)=\int_2^tf(u)du=F(t)~\mbox{(say)},$$
we substitute these in (\ref{eq:tmcmc1}) to obtain a much simplified expression for TMCMC simulation, given by
\begin{equation}
	h_1(t)=\left(li(t)+f(t)\right)\left(\frac{tF(t)}{\log t}\right)^{-k}.
	\label{eq:tmcmc2}
\end{equation}
For our TMCMC purpose, however, it may be necessary to evaluate $h_1(\cdot)$ at very large values of $t$, which may be a very burdensome and perhaps
numerically inaccurate exercise on ordinary computers. Hence,
we propose breaking up the sum in (\ref{eq:tmcmc2}) into two separate parts and carry out independent TMCMC simulations corresponding to the parts:
\begin{align}
	\tilde h_1(t)=li(t)\left(\frac{tF(t)}{\log t}\right)^{-k}~\mbox{and}~\tilde h_2(t)=f(t)\left(\frac{tF(t)}{\log t}\right)^{-k}
	\label{eq:tmcmc3}
\end{align}
Now, for TMCMC, we need to evaluate only the logarithms of $\tilde h_1(t)$ and $\tilde h_2(t)$ in (\ref{eq:tmcmc3}), which either completely bypasses or significantly
mitigates the problem of evaluating extremely large values.

Now suppose that we are interested in simulating primes $p$ that are greater than any given prime $p_0$, however large.
Then, assuming that we are given primes till $p_0$, 
we first simulate, by TMCMC, from (slightly abusing notation)
\begin{align}
\tilde h_1(t)=li(t+p_0)\left(\frac{(t+p_0)F(t+p_0)}{\log (t+p_0)}\right)^{-k}~\mbox{and}~\tilde h_2(t)=f(t+p_0)\left(\frac{(t+p_0)F(t+p_0)}{\log (t+p_0)}\right)^{-k};~t>0,
\label{eq:tmcmc4}
\end{align}
where $k$ is such that $k\log k\approx p_0$. 
For our purpose, we set 
\begin{equation}
F(x)=\frac{x}{\left(\log x\right)^{3/4}}\exp\left(-\sqrt{\frac{\log x}{6.315}}\right),
	\label{eq:F}
\end{equation}
so that
$$f(x)=F(x)\left(\frac{1}{x}-\frac{3}{4}\frac{1}{x\log x}-\frac{1}{2x\sqrt{6.315\log x}}\right).$$
%

In fact, with respect to (\ref{eq:tmcmc4}) we further consider the transformation $t\mapsto \exp(z)$, to get rid of the restriction $t>0$, to obtain
\begin{align}
	\breve h_1(z)=\tilde h_1(\exp(z))\exp(z)~\mbox{and}~\breve h_2(z)=\tilde h_2(\exp(z))\exp(z),
	\label{eq:tmcmc5}
\end{align}
where $z\in\mathbb R$.

For any TMCMC realization $z$ from (\ref{eq:tmcmc5}), we check if $[\exp(z)]+p_0$ is prime, using the Miller-Rabin algorithm (\ctn{Miller1976}, \ctn{Rabin1980}) 
for primality test. Note that this realization is not expected to be astronomically large, so that the Miller-Rabin test of primality are expected to work efficiently here.

It has come to our notice that since the latest largest Mersenne exponent $p_0=136,279,841$, GIMPS has so far tested primality of candidate Mersenne 
exponents till at least $140$ million, and none of them yielded positive result.
Hence, we attempt to discover, with our theory and methods, possible Mersenne exponents greater than $140$ million.
We implement our mixture of additive and multiplicative TMCMC with mixing probabilities $0.1$ and $0.9$, respectively, on an ordinary 64-bit laptop having $8$ GB RAM,
with i3-6100U CPU (2 physical cores), running at 2.30GHz.
The code, written in C, takes only several seconds to generate $10^7$ TMCMC realizations. Implementing the Miller-Rabin test of primality, also written in C,
for the realized forms $[\exp(z)]+p_0$, and repeating the TMCMC procedure several times, each time setting $p_0$ to be the largest prime obtained
in the last TMCMC run, or any large integer as desired, and updating $k$ appropriately, we obtained $259$ prime numbers in less than an hour. 
For most of the repetitions, the sets of primes obtained from $\breve h_1(z)$ and $\breve h_2(z)$ turned out to be the same. 

\subsection{Generation/identification of possible Mersenne prime exponents}
\label{subsec:identify_mersenne}
For any such prime $p+p_0$ obtained by TMCMC, it is then necessary to check if $2^{p+p_0}-1$ is prime, using the well-known and efficient Lucas-Lehmer primality test 
(\ctn{Lucas1878}, \ctn{Lehmer1930}) for Mersene numbers. But however efficient this test may be, for numbers this large, it is not possible to implement the test
on the machines of our institution. Indeed, even the highly specialized machines of GIMPS are expected to take months, even years, to arrive at the solutions.
One of the goals of this work is to offer a simple procedure to identify prime numbers that are good candidates for qualifying as Mersenne prime exponents.

The method needs to be devised in a way that avoids evaluation of $2^{p+p_0}-1$ if possible, as it requires arbitrary precision for such large $p+p_0$. Although Python and GMP
(GNU Multiple Precision) library in C offer arbitrary precision, such evaluations, with massive memory requirements, are not worthwhile unless exact answers with 
the Lucas-Lehmer primality test, are sought. Since, as already mentioned, the latter is anyway infeasible, we seek to create a method that bypasses evaluation of $2^{p+p_0}-1$.

Our method corresponds to a simple change-of-variable in (\ref{eq:tmcmc4}). Setting $s = 2^t-1$, we derive the (logarithm of the) 
distribution of $s$ from $\tilde h_1(t)$ of (\ref{eq:tmcmc4}), as
\begin{align}
	&\log[h^*_1(s)]=\log\left[li\left(\frac{\log(s+1)}{\log 2}+p_0\right)\right] \notag\\
	&\qquad-k\left[\log\left(\frac{\log(s+1)}{\log 2}+p_0\right)+\log \left\{F\left(\frac{\log(s+1)}{\log 2}+p_0\right)\right\}
	-\log\log\left(\frac{\log(s+1)}{\log 2}+p_0\right)\right]\notag\\
	&\qquad\qquad-\left(\frac{\log(s+1)}{\log 2}\right)\log 2,
	\label{eq:tmcmc6}
\end{align}
ignoring an additive constant.
Since $\left(\frac{\log(s+1)}{\log 2}\right)=t$, and as simulation of $s$ is equivalent to simulation of $t$ (and then taking $2^t-1$), we rewrite (\ref{eq:tmcmc6}) as
\begin{align}
	\log[h^*_1(s)]=\log\left[li\left(t+p_0\right)\right] 
	-k\left[\log\left(t+p_0\right)+\log \left\{F\left(t+p_0\right)\right\}-\log\log\left(t+p_0\right)\right]
	-t\log 2.
	\label{eq:tmcmc7}
\end{align}
As before, we consider the transformation $t\mapsto \exp(z)$ in (\ref{eq:tmcmc7}), where $-\infty<z<\infty$, to bypass the restriction $t>0$, to obtain (abusing notation),
\begin{align}
	&\log[h^*_1(z)]=\log\left[li\left(\exp(z)+p_0\right)\right]\notag\\ 
	&\qquad-k\left[\log\left(\exp(z)+p_0\right)+\log \left\{F\left(\exp(z)+p_0\right)\right\}-\log\log\left(\exp(z)+p_0\right)\right]\notag\\
	&\qquad\qquad -\exp(z)\log 2+z.
	\label{eq:tmcmc8}
\end{align}
It is clear from expression (\ref{eq:tmcmc8}) and the expression of $F(x)$ provided in (\ref{eq:F}), that evaluations of expressions like $2^t-1$ has been
completely avoided. However, the corresponding expression for $\tilde h_2(t)$ in (\ref{eq:tmcmc4}) would require such evaluations, but since our previous, general prime
simulations associated with both $\tilde h_1(t)$ and $\tilde h_2(t)$ yielded essentially the same set of primes, we zero in on (\ref{eq:tmcmc8}) for simulation of
candidate Mersenne prime exponents.

To implement, we set $p_0$ to be any prime simulated using our previous method corresponding to (\ref{eq:tmcmc4}), with $k$ corresponding to $p_0$. We then apply the same
TMCMC technique as before to (\ref{eq:tmcmc8}). Now note that in our previous implementation, we did not consider any burn-in period, as the goal was to generate primes
greater than $p_0$, and such primes, identified even during any burn-in period, still remain valid primes. However, here our goal is not just to generate primes,
but to generate primes that are candidate Mersenne exponents. Hence, we now consider a burn-in of $10^7$ iterations and accept those primes as valid
Mersenne candidates only if they are simulated in the next $10^7$ iterations. The primes generated using this procedure are clearly theoretically candidate Mersenne 
prime exponents, since they are generated from the posterior predictive distribution of $2^p-1$, $p$ being primes larger than $p_0$. 

Repeating this TMCMC procedure for (\ref{eq:tmcmc8}) with $p_0$ set as the primes obtained from (\ref{eq:tmcmc4}), we obtained $184$ candidate Mersenne prime exponents, 
out of the total $259$ primes generated from (\ref{eq:tmcmc4}), which we display in Tables \ref{table:mersenne1} and \ref{table:mersenne2}.
The number of digits of the potential Mersenne primes corresponding to the table entries range from $42,141,405$ to $242,429,718$.
Needless to mention, the Mersenne prime generation procedure can be continued without limit for flagging any desired number of potential Mersenne primes.
\begin{table}[h!]
\centering
	{\large \bfseries Candidate Mersenne prime exponents} \\[1em]
\begin{tabular}{|c|c|c|c|}
\hline
140000053 & 140000059 & 140000099 & 140000141 \\
140000177 & 140000213 & 140000257 & 140000261 \\
140000269 & 140000327 & 140000339 & 140000347 \\
140000351 & 140000353 & 140000369 & 140000387 \\
140000407 & 140000431 & 140000447 & 140000449 \\
140000459 & 140000477 & 141000037 & 141000043 \\
141000113 & 141000127 & 141000133 & 141000137 \\
141000149 & 141000161 & 141000163 & 141000173 \\
141000179 & 141000187 & 141000199 & 141000217 \\
141000253 & 141000257 & 141000259 & 142000049 \\
142000069 & 142000081 & 142000121 & 142000127 \\
142000139 & 143000021 & 143000029 & 143000047 \\
143000051 & 143000059 & 143000083 & 143000101 \\
143000113 & 143000149 & 143000159 & 143000161 \\
143000167 & 144000023 & 144000037 & 144000091 \\
144000097 & 144000113 & 144000119 & 144000133 \\
144000139 & 144000193 & 144000209 & 144000211 \\
144000223 & 145000045 & 145000049 & 145000057 \\
145000061 & 145000069 & 145000081 & 145000099 \\
145000117 & 145000183 & 145000187 & 145000199 \\
145000201 & 145000223 & 145000243 & 145000249 \\
145000279 & 145000291 & 145000307 & 145000351 \\
146000077 & 146000087 & 146000191 & 146000249 \\
146000269 & 146000273 & 146000287 & 146000291 \\
146000299 & 146000341 & 146000347 & 146000357 \\
146000399 & 146000411 & 146000419 & 146000429 \\
146000431 & 146000441 & 146000461 & 146000473 \\
\hline
\end{tabular}
\caption{First half, 92 elements.}
\label{table:mersenne1}
\end{table}

\begin{table}[h!]
\centering
{\large \bfseries Candidate Mersenne prime exponents (continued)} \\[1em]
\begin{tabular}{|c|c|c|c|}
\hline
146000501 & 146000507 & 146000509 & 146000513 \\
147000011 & 147000013 & 147000023 & 147000167 \\
147000181 & 147000193 & 147000253 & 147000257 \\
147000331 & 148000343 & 148000019 & 148000031 \\
148000043 & 148000049 & 148000051 & 148000117 \\
148000121 & 148000123 & 148000129 & 148000157 \\
148000159 & 148000163 & 148000183 & 148000189 \\
148000207 & 148000211 & 148000217 & 148000219 \\
148000247 & 148000267 & 148000277 & 148000297 \\
148000301 & 148000343 & 149000021 & 149000101 \\
149000107 & 149000123 & 149000129 & 149000143 \\
149000153 & 149000197 & 149000213 & 149000219 \\
149000237 & 149000251 & 149000263 & 149000281 \\
149000287 & 159000157 & 159000253 & 159000263 \\
159000349 & 200200003 & 200200009 & 200200019 \\
200200031 & 220000013 & 230000003 & 300000007 \\
300000031 & 397516723 & 397516733 & 397516753 \\
397516759 & 397516781 & 397516783 & 705032701 \\
705032719 & 705032723 & 805032703 & 805032733 \\
\hline
\end{tabular}
\caption{Second half, 92 elements.}
\label{table:mersenne2}
\end{table}



\section{Summary and conclusion}
\label{sec:conclusion}
The PNT brings with it the fragrance of randomness, which has attracted our attention for further unravelling the delightful mystery of prime numbers 
in the new light of Bayesian statistics. As we showed, the PNT itself strongly motivates an NHPP for the prime counts, which we thankfully exploited
to create a distribution of the prime numbers, using the associated inter-arrival time distribution. 
Importantly, we established the following: (a) that our NHPP is strongly consistent with respect to the PNT; (b) that it is asymptotically equivalent to large prime numbers 
and (c) a fundamentally important result on prime gaps using our NHPP.
Intimately tied to the precise PNT and our distribution theory,
is RH, the most elusive mathematical conjecture ever, yet unsolved and openly challenging all the mathematicians of the world. Do Bayesians stand a chance? As already
mentioned, past investigation of these authors using a novel Bayesian characterization of infinite series, provided strong evidence against RH. In this article,
it is the PNT that gives us the second chance to once again refute RH, with respect to a completely new direction. Indeed, the distribution of the prime numbers,
in conjunction with priors for the model parameters, yields a Bayesian model, which we analyse asymptotically, using existing theory of prime numbers, to arrive at the
conclusion that RH is false with respect to our Bayesian model. Interestingly, two different statements on RH, connected to the PNT, both yielded the same Bayesian conclusion.
It is noteworthy to point out once again that our Bayesian theory is robust with respect to the prior parameters, namely, $a,b,\gamma,\xi$, since none of our 
asymptotic results depend upon their chosen values. Even improper priors with $a=b=\gamma=\xi=0$, yield the same asymptotic results. What is more interesting is that,
the fast and efficient recursive Bayesian theory that we developed for prediction of large prime numbers, actually yields the same asymptotic inference as
the traditional, but practically infeasible, non-recursive Bayesian theory. To remind again, the non-recursive Bayesian posteriors of the parameters and posterior
predictive distributions consist of the order of $2^{\mbox{million}}$ terms when conditioned upon the first million primes. On the other hand, the
corresponding number of terms for the recursive Bayesian idea are just $2$ and $4$! 

Beyond its theoretical contributions, the proposed framework yields a practical prime-hunting algorithm based on TMCMC.
Through posterior predictive simulation, and a simple change-of-variable facilitating the generation of Mersenne prime exponent candidates, we demonstrated the 
feasibility of identifying very large primes using modest computational resources. Our implementation successfully uncovered 259 primes exceeding 140 million, 
including 184 strong Mersenne candidates, corresponding to potential Mersenne primes with tens to hundreds of millions of digits.


The methodology presented here highlights the power of Bayesian modeling in number theory. By unifying Bayesian formalism, efficient computation, and rigorous 
inference, it provides both novel insights into the structure of primes and practical tools for their exploration. Future research may extend this framework to related 
number-theoretic problems and further develop its computational capabilities.

\newpage

\renewcommand\baselinestretch{1.3}
\normalsize
\bibliography{irmcmc}


\end{document}